\documentclass[12pt]{amsart}

\usepackage{amssymb}
\usepackage{amsmath}
\usepackage{amsthm}
\usepackage{enumerate}

\usepackage{xcolor}

\usepackage{tikz}
\usepackage[T1]{fontenc}
\usetikzlibrary{patterns}
\tikzset{
	convexset/.style = {line width = 0.5 pt, fill=lightgray, opacity=0.3},
	ext/.style = {circle, inner sep=0pt, minimum size=2pt, fill=black},
	segment/.style = {line width = 0.75 pt}
}

\theoremstyle{plain}

\newtheorem{theorem}{Theorem}[section]

\newtheorem{lemma}[theorem]{Lemma}
\newtheorem{proposition}[theorem]{Proposition}
\newtheorem{corollary}[theorem]{Corollary}

\newtheorem{observation}[theorem]{Observation}
\newtheorem{fact}[theorem]{Fact}

\theoremstyle{definition}

\newtheorem{definition}[theorem]{Definition}
\newtheorem{remark}[theorem]{Remark}

\newtheorem{notation}[theorem]{Notation}

\newcommand{\R}{\mathbb{R}}
\newcommand{\N}{\mathbb{N}}

\newcommand{\inte}{\mathrm{int}\:}

\renewcommand{\epsilon}{\varepsilon}
\renewcommand{\phi}{\varphi}
\renewcommand{\tilde}{\widetilde}

\author[C.A.~De~Bernardi]{Carlo Alberto De Bernardi}

\address{Dipartimento di Matematica per le Scienze economiche, finanziarie ed attuariali, Universit\`a Cattolica del Sacro Cuore, 20123 Milano,Italy}

\email{carloalberto.debernardi@unicatt.it}
\email{carloalberto.debernardi@gmail.com}

\author[L.~Vesel\'y]{Libor Vesel\'y}
\address{Dipartimento di Matematica\\
	Universit\`a degli Studi\\
	Via C.~Saldini 50\\
	20133 Milano\\
	Italy}

\email{libor.vesely@unimi.it}



\thanks{}
\title{Extendability of quasiconvex functions\\ in finite-dimensional normed spaces}

\begin{document}

\begin{abstract}
Let $X$ be a normed space of a finite dimension at least two, and $C\subsetneq X$ a closed convex set
with nonempty interior. We are interested in extending Lipschitz quasiconvex functions on
$C$ to quasiconvex functions on $X$. We show that, unlike what holds for convex functions,
in general one cannot obtain Lipschitz extensions (except for trivial cases). If we require
just uniformly continuous or continuous extensions, such extendability properties
for $C$ are shown to be characterized by some geometric properties of $C$.
\end{abstract}

\maketitle

%%%%%%%%%%%%%%%%%%%%%%%%%%%%%%%%%%%%%%%%%%%%%%%%%%%%%%

%%%%%%%%%%%%%%%%%% Headings
\markboth{C.A.~De Bernardi and L.~Vesel\'y}{Extending quasiconvex functions from uniformly convex sets}
%%%%%%%%%%%%%%%%%%%%%%%%%%%

\section{Introduction}

Let $C$ be a convex set in a real vector space. A function
$f\colon C\to \R$  
is said to be {\em quasiconvex} (QC, for short)
if all its sub-level sets $\{x\in C: f(x)\le\alpha\}$ 
[equivalently, all its strict sub-level sets $\{x\in C: f(x)<\alpha\}$] ($\alpha\in\R$)
are convex. QC functions represent a natural generalization of convex functions and
play a crucial role in Optimization, in Mathematical programming, in Mathematical economics,  and in many other areas 
of mathematical analysis  (see \cite{BFitzV,gencon,pierskalla,penot} and the references therein).

In the present paper and in our two previous papers \cite{DEVEqc_examples,DEVEqc_UC}, we are interested in extendability of Lipschitz or (uniformly) continuous 
QC functions defined on a convex body in a real normed space $X$ of dimension at least two. By a {\em convex body}
we mean a (possibly unbounded) closed convex set with nonempty interior which is not the whole space.
Our motivation comes from the following well-known extension result valid for {\em convex} functions.
(Its metric-space version, without any convexity, is due to E.J.~McShane \cite{McShane}, 
while the ``convex version''
belongs to the mathematical folklore.)

\begin{theorem}\label{T:McS}
Let $C\subset X$ be an arbitrary nonempty convex set, and let $f\colon C\to\R$ be a convex function
which is $L$-Lipschitz (that is, 
Lipschitz with a Lipschitz constant $L\ge0$). Then the formula
$F(x)=\inf_{u\in C}\bigl\{ f(u)+L\|x-u\| \bigr\}$ defines an $L$-Lipschitz convex
extension of $f$ to the whole $X$.
\end{theorem}

It is natural to ask whether some analogous result holds for QC functions. It turns out
that the ``quasiconvex'' situation is completely different and depends heavily on
geometric properties of the convex set $C$. Here we mainly consider the case in which $C$ is a convex body. 
In the following theorem, we collect our main previous results, contained in \cite{DEVEqc_examples,DEVEqc_UC}. The natural definitions
of a UR (uniformly rotund) body and an LUR (locally uniformly rotund) body can be found
in the cited papers.

\begin{theorem}\label{T:old}
Let $X$ be a real normed space of dimension at least two.
\begin{enumerate}[(a)]
\item There exists a Lipschitz QC function $f$ on the closed unit disc $D\subset\R^2$
which does not admit any Lipschitz QC extension to the whole $\R^2$ 
{\em (see \cite[Example~4.2]{DEVEqc_examples}).}
Consequently,%
\footnote{Indeed, one can consider $X$ in the form of the direct sum $X=\R^2\oplus_\infty Y$, and take $C=B_X$, the closed unit
ball of such $X$. Let $f_1\colon B_{\R^2}\equiv D\to\R$ be a Lipschitz quasiconvex
function which
cannot be extended to a Lipschitz QC function on $\R^2$. Then
the function $f\colon C\to\R$, given by $f(x)=f_1(Px)$ where $P$ is the canonical projection
of $X$ onto $\R^2$, is the desired example.}
$X$ contains a convex body $C$ such that
some Lipschitz QC function on $C$ cannot be extended to a Lipschitz QC function
on the whole $X$.
\item If $X$ is a Banach space and $C\subset X$ is a convex body which is not LUR, there exists a 
Lipschitz QC function
$f\colon C\to\R$ which does not admit any continuous QC extension to the whole $X$ 
{\em (see \cite[Theorem~3.3]{DEVEqc_examples})}.
\item If $C\subset X$ is a UR convex body, then each uniformly continuous QC function
$f\colon C\to\R$ admits a uniformly continuous QC extension to $X$ 
{\em (see \cite[Theorem~5.5]{DEVEqc_UC}).}
\end{enumerate}
\end{theorem}

It remains an {\em open problem} whether the UR property of a convex body $C$ 
is also necessary for extendability
of each uniformly continuous QC function on $C$ to a uniformly continuous QC function
on the whole $X$. In other words, we don't know what happens for convex bodies which are 
LUR but not UR.

However, in the finite-dimensional case,  we obtain exact characterizations of various types
of extendability in terms of geometric properties of $C$.
And these results are the subject of the present paper.

The paper is organized as follows. After a brief section of preliminaries, the next two Sections~3 and~4
provide three (finite-dimensional) constructions of Lipschitz QC functions that cannot be extended to
sufficiently good QC functions defined on the whole space (Theorems~\ref{noQC}, \ref{T:nonR},
and~\ref{no_ucQC}). In Section~5, we show that on every convex body in an arbitrary normed space of dimension at least 2
there exists a Lipschitz QC function which has no Lipschitz QC extension to the whole space 
(Theorem~\ref{no_lipQC}). Section~6 contains some technical results concerning
convex bodies with no asymptotic directions, which will be used
in Section~7 to prove Theorem~\ref{T:ext_contQC} on extendability of QC functions with preserving
continuity. The subsequent Section~8 concerns extendability of upper semicontinuous QC functions.
Finally, in the last Section~9 we combine results from previous sections and some
our earlier results to obtain several full
(finite-dimensional) characterizations of extendability
of QC functions from a closed convex set, in terms of geometric properties of the latter.

%%%%%%%%%%%%%%%%%%%%%%%%%%%%%%
\medskip

\section{Preliminaries}

Let $\R_+=[0,+\infty)$.
Throughout the paper, if not specified otherwise, $X$ denotes a (real) normed linear space of {\em finite} dimension $d\ge2$. 
By $B_X$, $U_X$ and $S_X$ we denote the closed unit ball, the open unit ball and the unit sphere
of $X$, respectively. For $x\in X$ and $r>0$ we denote $B(x,r):=x+r B_X$ (the closed ball
of radius $r$ around $x$).

Given two sets $A,B\subset X$, their {\em distance} is the quantity
$d(A,B)=\inf\{\|a-b\|: a\in A, b\in B\}$, with the usual convention that $\inf\emptyset=+\infty$.
The distance of a point $x\in X$ from $A$ is defined as $d(x,A):=d(\{x\},A)$. 
If $A\subset B\subset X$, we denote
by $\mathrm{int}_B A$, $\partial_B A$ and $\overline{A}^B$ the relative interior of $A$ in $B$,
the relative boundary of $A$ in $B$ and the relative closure of $A$ in $B$, respectively.
In the particular case of $B=X$, the subscript or superscript $X$ will be omitted. If not stated otherwise, 
$\R^d$ ($d\in\N$) is endowed with the euclidean norm. 

We shall use the following standard fact about convex sets.

\begin{fact}\label{F:convex}
Let $C$ be a convex set with nonempty interior in an arbitrary normed space $X$.
Then $\mathrm{int}\,C=\mathrm{int}\,\overline{C}$ and $\overline{C}=\overline{\mathrm{int}\,C}$.
\end{fact}

Further, we are going to use the following {\em simplified notation}. If $f$ is a real-valued function
defined (at least) on a set $E$, we denote, for instance,
$$
[f<\alpha]_E:=\{x\in E: f(x)<\alpha\}\quad\text{for $\alpha\in\R$.}
$$
In this notation, if the set $E$ is convex, $f$ is a {\em quasiconvex function}
(QC function, for short) on $E$ if and only 
if all such sets are convex.

By
a {\em convex body} in $X$ we mean a closed, convex, proper subset $C\subset X$ 
with nonempty interior (the same definition applies also in the infinite-dimensional case). 
A convex body is said to be {\em rotund} (or strictly convex) if 
its boundary contains no nontrivial line segments; equivalently, if the midpoint of two
distinct elements of $\partial C$ (of $C$, equivalently) is always an interior point of $C$.

The {\em recession cone} of a nonempty convex set
$C\subset X$ is the set
$$
\mathrm{recc}(C)=\{v\in X: C+\R^+v=C\}=\{v\in X:\; c+\R^+v\subset C\;\text{for some $c\in C$}\}.
$$
For a nonempty closed convex set $C$ it is well known (see \cite{Rock}) that:
\begin{itemize}
\item $\mathrm{recc}(C)$ is a closed convex cone with vertex at the origin;
\item $C$ is bounded if and only if $\mathrm{recc}(C)=\{0\}$;
\item if $C$ has nonempty interior then $\mathrm{recc}(C)=\mathrm{recc}(\mathrm{int}\,C)$.
\end{itemize}

\begin{definition}
Let $C\subsetneq X$ be a closed convex set, and $v\in X\setminus\{0\}$. We shall say that
$v$ is an {\em asymptotic direction} for $C$ if there exists $x_0\in X$ such that
\begin{equation}\label{D:ad}
(x_0+\R^+ v)\cap(\mathrm{int}\,C)=\emptyset
\quad\text{and}\quad \lim_{t\to+\infty}d(x_0+tv,C)=0.
\end{equation}
\end{definition}

\begin{observation}
Let $v$ be an asymptotic direction of a (necessarily unbounded) 
closed convex set $C\subsetneq X$, and let $x_0\in X$ be as in \eqref{D:ad}. 
Then:
\begin{enumerate}[(a)]
\item $v\in\mathrm{recc}(C)$;
\item $(x_0+\R v)\cap(\mathrm{int}\,C)=\emptyset$;
\item either $\partial C$ contains a half-line of the form $c_0+\R_+v$ or
$(x_0+\R v)\cap C=\emptyset$.
\end{enumerate}
\end{observation}

\begin{proof}
We omit the easy proof of (b) and (c). Let us show (a).
For each $t>0$, let $c_t\in C$ be such that $\|x_0+tv-c_t\|=d(x_0+tv,C)$. 
Since $c_t=tv+x_0+(c_t-x_0-tv)$, we have $\frac1t c_t\to v$ as $t\to+\infty$.
By \cite[Theorem~8.2]{Rock}, $v\in\mathrm{recc}(C)$.
\end{proof}

A convex set $A\subset X$ will be called {\em relatively open} if it is open in its affine hull.

\medskip

\section{Two elementary negative results}

The aim of the present section is to provide two non-extendability results. These results are
quite elementary.

\begin{theorem}\label{noQC}
Let $\mathrm{dim}\,X=d\ge2$, and let $C\subset X$ be an unbounded convex body
that admits an asymptotic direction. Then there exists a Lipschitz QC function
$f\colon C\to\R$ such that $f$ admits no QC (whatsoever) extension to the whole
$X$.
\end{theorem}

\begin{proof}
We can (and do) assume that $0\in\mathrm{int}\,C$. Let $v$ be an asymptotic direction for $C$,
hence there exists $x_0\in X$ such that the line $x_0+\R v$ does not intersect $\mathrm{int}\,C$
and $d(x_0+tv,C)\to 0$ as $t \to+\infty$. Notice that $x_0$ and $v$ are linearly independent.

By the Hahn-Banach theorem,
there exists $h\in X^*\setminus\{0\}$ such that $\sup h(C)=\inf h(x_0+\R v)\equiv h(x_0)=1$.
Clearly, $h(v)=0$. Let $Z\subset X$ be a (necessarily $(d-2)$-dimensional) subspace
such that $\mathrm{Ker}(h)=\R v\oplus Z$. So, we have
$$
X=Y\oplus Z \quad \text{where} \quad Y=\mathrm{span}\{x_0,v\}.
$$
By means of an appropriate linear homeomorphism of the form
$$
tv+s x_0+z \mapsto (t,s,\zeta)\,,\quad t,s\in\R,\ z\in Z,\ \zeta\in\R^{d-2},
$$
we can (and do) identify $X$ with the Euclidean space $\R^d\equiv\R\times\R\times\R^{d-2}$.
After this identification, we have
$$
C\subset\{(t,s,\zeta): s\le1\} ,\quad
\R^+v\subset\mathrm{int}\,C  \quad\text{and}\quad
d((t,1,0),C)\to0\ \text{as $t\to+\infty$.}
$$

Put $A=\{(t,s): s<1\}\subset\R\times\R$. Let us inductively define real numbers $b_n>0$ and
open convex sets $D_n\subset A$ ($n\ge1$ integer).
{\em The reader is encouraged to sketch a picture!}

Fix a decreasing sequence $\{\epsilon_n\}_n\subset(0,+\infty)$ such that $\epsilon_n\to0$.
Put $b_1=1$ and define
$$\textstyle
B_1=\left\{
(t,s):\; s<1,\;s<1-\frac{\epsilon_1}{b_1}(t-b_1)
\right\},
$$
that is, the part of $A$ lying below the line connecting the points
$(0,1+\epsilon_1)$ and $(b_1,1)$.
Now, assume we have already defined $b_n>0$ and an open convex $B_n\subset A$.
Then choose $b_{n+1}>0$ so that $(b_{n+1},0)\notin 2 B_n$, and define
$$\textstyle
B_{n+1}=\left\{
(t,s):\; s<1,\;s<1-\frac{\epsilon_{n+1}}{b_{n+1}}(t-b_{n+1})
\right\},
$$
that is, the part of $A$ lying below the line connecting the points
$(0,1+\epsilon_{n+1})$ and $(b_{n+1},1)$. 

By our construction, the sequence $\{B_n\}_{n\ge1}$ is increasing; 
moreover,
$$
\bigcup_{n\ge1}B_n =A,\ \text{and}\ 
A\cap 2B_n\subset B_{n+1}\ \text{for each $n\ge1$.}
$$
By \cite[Lemma~2.9]{DEVEqc_examples}, there exist an unbounded, increasing
sequence $\{\alpha_n\}_{n\ge1}$ and a Lipschitz QC function
$g\colon A\to\R$ such that $B_n=[g<\alpha_n]_A$ for each $n$. Clearly,
$g$ admits a unique Lipschitz (necessarily QC) extension to 
$\overline{A}=\{(t,s): s\le1\}$; so let us consider $g$ defined on $\overline{A}$.

Consider the canonical projection $P\colon\R^d\to\R^2$ defined by
$$
P(t,s,\zeta)=(t,s),\quad t,s\in\R,\ \zeta\in\R^{d-2}.
$$
Now, the formula
$$
f(t,s,\zeta):=g(P(t,s,\zeta))\,, (t,s,\zeta)\in C,
$$
gives a well-defined Lipschitz QC function $f\colon C\to\R$.
We claim that $f$ cannot be extended to a QC function
defined on the whole $\R^d$.

Proceeding by contradiction, assume that there exists a QC extension
$F$ of $f$ to the whole $\R^d$. For each $n\ge2$,
$[F<\alpha_n]_{\R^d}\cap (\mathrm{int C})=
[f<\alpha_n]_{\mathrm{int}\,C}=P^{-1}([g<\alpha_n]_A)\cap(\mathrm{int}\,C)$.
Since the boundary of the last set contains a relatively open piece of
the hyperplane $\{(t,s,\zeta): s=1-\frac{\epsilon_{n+1}}{b_{n+1}}(t-b_{n+1})\}$,
by convexity the set $[F<\alpha_n]_{\R^d}$ must be contained in the
half-space $\{(t,s,\zeta): s<1-\frac{\epsilon_{n+1}}{b_{n+1}}(t-b_{n+1})\}$.
But this implies that $F(0,1+\epsilon_1,0)\ge\alpha_n$ ($n\ge2$). This is clearly 
impossible since $\alpha_n\nearrow+\infty$.
\end{proof}

%%%%%%%%%%%%%%%%%%%%%%%%%%%%%%%%%%%%%%%%%
\medskip

The next result follows from our general result from \cite{DEVEqc_examples},
described in Theorem~\ref{T:old}(b). Here we briefly sketch 
a simple direct proof that follows the line of the proof of Theorem~\ref{noQC}.

\begin{theorem}\label{T:nonR}
Let $\mathrm{dim}(X)=d\ge2$ and let $C\subset X$ be a convex body which is not rotund.
Then there exists a Lipschitz QC function $f\colon C\to\R$ that admits no continuous 
QC extension to the whole $X$.
\end{theorem}

\begin{proof}
Assume that $\partial C$ contains a nontrivial line segment $[c,d]$. Let $0\in\mathrm{int}\,C$ and
$v=\frac12(d-c)$. By the Hahn-Banach theorem, there exists $h\in X^*\setminus\{0\}$
such that $h(v)=0$ and $h(c)\equiv\inf h([c,d])=\sup h(C)=1$. Write $\mathrm{Ker}(h)=\R v \oplus Z$,
so that $X=\R v\oplus\R c\oplus Z$. We can suppose that 
$X=\R\times\R\times\R^{d-2}\equiv\R^d$, $C\subset\{(t,s,\zeta):s\le1\}$
and $\partial C$ contains the segment with endpoints $(0,1,0)$ and $(2,1,0)$.
Consider strictly monotone sequences $\epsilon_n\searrow 0$ and $b_n\nearrow 2$ such that $b_1=1$.
In $A=\{(t,s): s<1\}$, define the sets $B_n$ ($n\ge1$) by 
$$\textstyle
B_{n}=\left\{
(t,s):\; s<1,\;s<1-\frac{\epsilon_{n}}{b_{n}}(t-b_{n})
\right\}
$$
as in the previous theorem. That is, $B_n$ is the part of $A$ lying under
the line that connects $(0,1+\epsilon_n)$ and $(b_n,1)$.
Then the sets $B_n$ form an increasing sequence,
and for each $n$ there exists $\eta_{n+1}>1$ such that $A\cap \eta_{n+1}B_n\subset B_{n+1}$.
By \cite[Lemma~2.9]{DEVEqc_examples}, there exist an increasing sequence
$\{\alpha_n\}_n$ of positive reals and a Lipschitz QC $g\colon \overline{A}\to\R$
such that $[g<\alpha_n]_A=B_n$ for each $n$. 
It is easy to see that $g(0,s,0)<\alpha_1$ whenever $0<s<1$,
and $g(0,1,0)\le\alpha_1$.
Finally, define $f\colon C\to\R$ by $f=g\circ P$, where 
$P\colon(t,s,\zeta)\to(t,s)$ for $(t,s,\zeta)\in\R\times\R\times\R^{d-2}$.

Assume there exists a continuous QC extension $F\colon\R^d\to\R$
of $f$. By the same convexity argument as in the proof of Theorem~\ref{noQC},
for each fixed $n\ge1$, we must have $F(0,1+\epsilon_n,0)\ge \alpha_{k}$ for each $k>n$.
Now, denoting $\alpha_\infty=\sup_n\alpha_n$ we obtain
$F(0,1+\epsilon_n,0)\ge\alpha_\infty$ for each $n$. On the other hand,
$F(0,1,0)=f(0,1,0)\le\alpha_1<\alpha_\infty$, which is impossible since
$F$ should be continuous.
\end{proof}

%%%%%%%%%%%%%%%%%%%%%%%%%%%%%%%%%%%%%%%%%%%%%
\medskip

\section{Nonexistence of uniformly continuous QC extensions}

In our finite-dimensional case, the result from \cite{DEVEqc_UC} described in
Theorem~\ref{T:old}(c) implies
that {\em if $C\subset X$ is a bounded, rotund convex body, then every uniformly continuous QC function
on $C$ admits a (uniformly) continuous QC extension to the whole $X$.} (Indeed, thanks to compactness,
such $C$ is automatically uniformly rotund.) Our present Theorem~\ref{no_ucQC} shows
that the boundedness assumption cannot be omitted. The proof is quite technical, but very geometric in its nature.

\begin{theorem}\label{no_ucQC}
Let $\mathrm{dim}\,X=d\ge2$, and let $C\subset X$ be an unbounded, rotund, convex body
with no asymptotic directions.
Then there exists a Lipschitz QC function $f\colon C\to\R$ that admits no
uniformly continuous QC extension to the whole $X$.
\end{theorem}

\medskip

The rest of the present section will be devoted to the proof. We shall proceed in several steps.

\medskip

\begin{proof}[Proof of Theorem~\ref{no_ucQC}]
First, we may (and do) assume that $0\in\mathrm{int}\,C$. Let $R>0$ be such that 
$$
R B_X\subset C.
$$
Since $C$ is unbounded, there exists a nonzero $v\in\mathrm{recc}(C)$; consequently, $\R^+ v\subset\mathrm{int}\,C$.
Since $C$ cannot contain lines (otherwise also $\partial C$ would), the line $\R v$ intersects $\partial C$
at exactly one point, say $c_0$. Let $h\in S_{X^*}$ be such that $m:=h(c_0)=\min h(C)$ ($<0$). By
taking an appropriate multiple of $v$, we can (and do) assume that $h(v)=1$.

\medskip

\noindent\textbf{\em Claim 1.} For each $\alpha\ge m$, the set $C_\alpha:=[h\le\alpha]_C$ is bounded.

\smallskip\noindent
If this is not the case, for some $\alpha\ge m$ we have
$C_\alpha\subset[m\le h\le\alpha]_X$ and the last set contains a half-line.
This forces $h$ to be constant on such a half-line, and therefore
$\partial C$ contains a half-line parallel to the previous one and passing
throuh $c_0$. But this is impossible since $C$ is rotund. Claim~1 is proved.

\medskip

Now, fix an arbitrary $u\in\mathrm{Ker}(h) \setminus\{0\}$ and define $Y=\mathrm{span}\{u,v\}$.
Clearly, $\mathrm{dim}\,Y=2$ and $c_0\in Y$. Let $Z\subset\mathrm{Ker}(h)$ be a subspace such that
$\mathrm{Ker}(h)=\R u\oplus Z$. Consequently, $X=\R u\oplus \R v\oplus Z=Y\oplus Z$.
Let $P\colon X\to Y$ be the onto linear projection along $Z$, that is, $\mathrm{Ker}(P)=Z$.
Notice that $P(\mathrm{Ker}(h))=\R u$ and $h=h\circ P$ (that is, $h(\alpha u+\beta v+z)=\beta=h(\alpha u+\beta v)$
whenever $\alpha,\beta\in\R$, $z\in Z$).

\medskip

\noindent\textbf{\em Claim 2.} The set $D:=P(C)$ is an unbounded, rotund, convex body in $Y$.

\smallskip\noindent
$D$ is clearly unbounded since it contains $\R^+ v$. Since $P$ is an open mapping,
$D$ has nonempty interior. Moreover, $D\subset[h\ge m]_Y$ implies that
$D\ne Y$. To show that $D$ is closed in $Y$, it suffices to show that
each set $D_\alpha:=[h\le\alpha]_D$ is closed, but this is clear since
$D_\alpha=P(C_\alpha)$ and $C_\alpha$ is compact. Finally,
if $y,y'\in\partial D$ are two distinct points, then
there exist unique (distinct) points $x,x'\in\partial C$ such that
$Px=y$ and $Px'=y'$. 
Then $\frac12(y+y')=P(\frac12 (x+x'))\subset P(\mathrm{int}\,C)\subset\mathrm{int}\,D$.
This shows that $D$ is rotund. Claim~2 is proved.

\medskip

	Fix one of the two connected components of $[h\ge0]_{\partial D}$, say
$\Gamma$. Since $\min h(D)=m<0$ and $R B_Y\subset D$, $\Gamma$ can be seen
as the graph of an increasing, Lipschitz, strictly concave function $g$ on $\R^+$. Therefore,
the restriction $h|_\Gamma$ provides a bi-Lipschitz homeomorphism between $\Gamma$ and $\R^+$.
So there exists $\beta>0$ such that
$$
\beta\|y-y'\| \le |h(y)-h(y')| \le \|y-y'\|,\quad y,y'\in\Gamma.
$$
Let $\{y_n\}_{n\ge1}\subset\Gamma$
be the sequence defined by:
$$
0=h(y_1)<h(y_2)<h(y_3)<\cdots\quad\text{and}\quad
\|y_{n+1}-y_n\|=1\ \text{for each $n$.}
$$
Let us also denote, for $k\in\N$, $\alpha_k=h(y_k)$ and notice that by our previous observation we have 
$\alpha_{k+1}-\alpha_k\ge\beta$.
\medskip

\noindent\textbf{\em Claim 3.}\ \ $\| y_{n+1}+y_{n-1}-2y_n  \|\to 0$.
\smallskip

\noindent
The sequence $\Delta_n:=y_{n+1}-y_n$ is contained in the half-sphere $[h>0]_{S_Y}$, 
or more precisely, in one of the two connected components, say $\gamma$,
of the set $[h>0]_{S_Y}\setminus\bigl\{v/\|v\|\bigr\}$.
By strict concavity of $g$ (or rotundity of $D$), the sequence $\{\Delta_n\}_n$ is strictly
monotone in any of the two natural orderings of $\gamma$ (clockwise or counter-clockwise).
Because of finite length of $\gamma$, for every fixed $\epsilon>0$ there can be only
finitely many indices $n\ge2$ such that $\|\Delta_n-\Delta_{n-1}\|\ge\epsilon$. In other words,
$\|\Delta_{n}-\Delta_{n-1}\|\to 0$, which is our Claim~3.
\smallskip

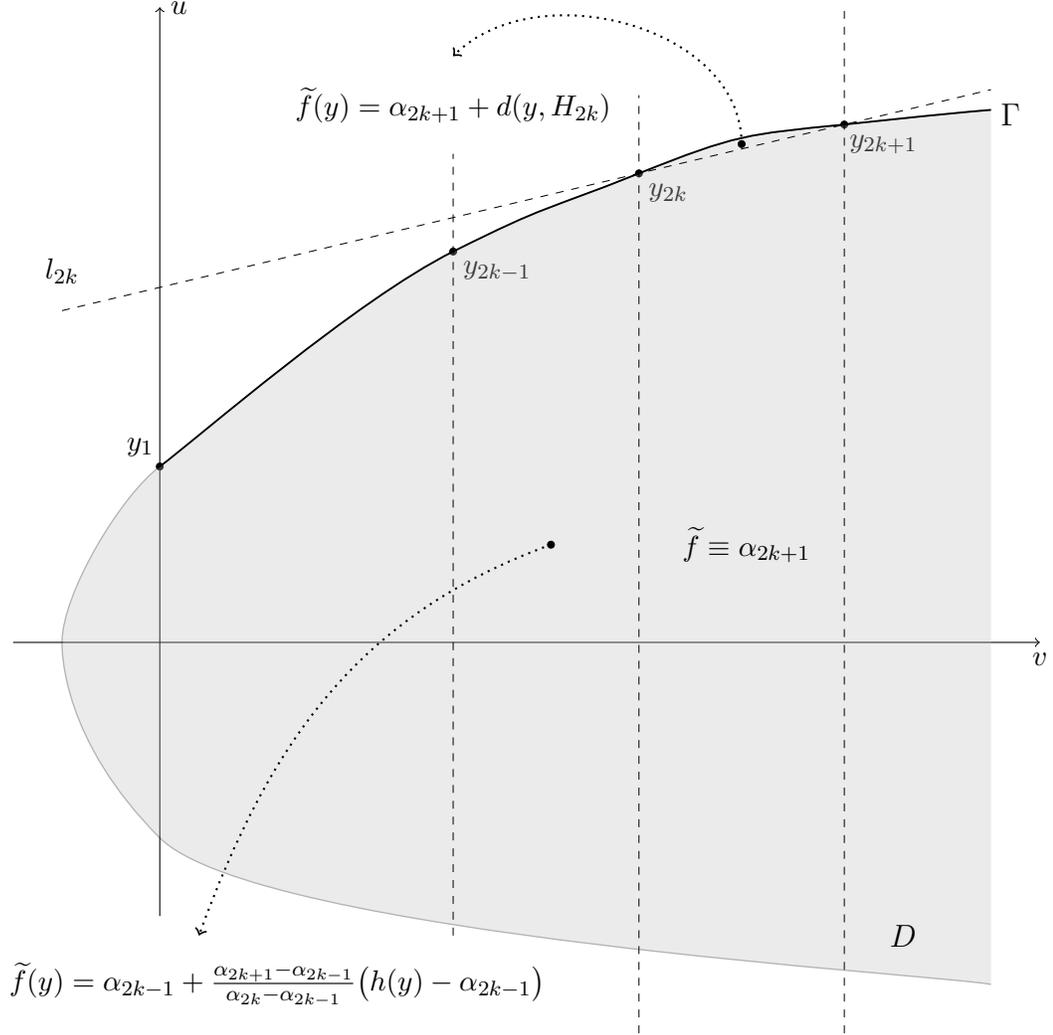
\begin{figure}
	\centering

	\begin{tikzpicture}[scale=1.3]
		\draw[->](-1.5,0)--(9,0)node[below]{\small{\( v \)}};
		\draw[->](0,-2.8)--(0,6.5)node[right]{\small{\( u \)}};

		\draw (0,1.8) circle[radius=1pt];
		\fill (0,1.8) circle[radius=1pt];
		\node at (-.2,2) {\small{$y_1$}};

		\draw (3,4) circle[radius=1pt];
		\fill (3,4) circle[radius=1pt];
		\node at (3.45,3.8) {\small{$y_{2k-1}$}};

		\draw (4.9,4.8) circle[radius=1pt]; \fill (4.9,4.8) circle[radius=1pt]; \node at (5.2,4.6) {\small{$y_{2k}$}};

		\draw (7,5.3) circle[radius=1pt];
		\fill (7,5.3) circle[radius=1pt];
		\node at (7.4,5.1) {\small{$y_{2k+1}$}};

		\draw[convexset]
		(8.5,-3.5) .. controls + (-0.5,0.1) and + (1,-1) ..(0,-2) .. controls + (-1,1) and + (0,-0.1) .. (-1,0)  .. controls + (0,0.5) and + (-0.4,-0.3) ..(0,1.8) .. controls + (0.4,0.3) and + (-1,-0.5) .. (3,4) .. controls + (1,0.5) and + (-1,-.4) ..(4.9,4.8).. controls + (1,.4) and + (-1,-.1) ..(7,5.3).. controls + (.1,.01) and + (-.6,-0.05) ..(8.5,5.45);

		\draw[line width = .7 pt](0,1.8) .. controls + (0.4,0.3) and + (-1,-0.5) .. (3,4) .. controls + (1,0.5) and + (-1,-.4) ..(4.9,4.8).. controls + (1,.4) and + (-1,-.1) ..(7,5.3).. controls + (.1,.01) and + (-.6,-0.05) ..(8.5,5.45);
		
		\node at (8.7,5.4) {{$\Gamma$}};	
		\node at (7.6,-3) {{$D$}};	
		
		\draw[dashed,thick, line width = 0.2 pt](-1,3.3952)--(8.5,5.6571);
		
		\node at (-1,3.8) {\small{$l_{2k}$}};

		\draw[dashed,thick, line width = 0.2 pt](3,-3)--(3,5);
		\draw[dashed,thick, line width = 0.2 pt](4.9,-4)--(4.9,5.6);
		\draw[dashed,thick, line width = 0.2 pt](7,-4)--(7,6.5);

		\node at (6,1) {\small{$\tilde{f}\equiv 
				\alpha_{2k+1}$}};

		\draw [->, thick, dotted, rounded corners] (4,1) to [out=200,in=70] (0.4,-3);
		\draw (4,1) circle[radius=1pt]; \fill (4,1) circle[radius=1pt];

		\node at (1.2,-3.5) {\small{$\tilde{f}(y)= \alpha_{2k-1}+\frac{\alpha_{2k+1}-\alpha_{2k-1}}{\alpha_{2k}-\alpha_{2k-1}}\bigl(h(y)-\alpha_{2k-1}\bigr)$}};

		\draw [->, thick, dotted, rounded corners] (5.95,5.1) to [out=90,in=45] (3,6);
		\draw (5.95,5.1) circle[radius=1pt]; \fill (5.95,5.1) circle[radius=1pt];

		\node at (3,5.5) {\small{$\tilde{f}(y)=\alpha_{2k+1} +d(y,H_{2k})$}};

	\end{tikzpicture}
	
	\caption{Construction of the function $\tilde f$}\label{FIG-ftilde}
\end{figure}

\noindent{\em Definition of $\tilde f\colon D\to\R$ (see Figure~\ref{FIG-ftilde}).} 
\smallskip

For $k\in\N$, let $l_k$ be the line in $Y$ passing through the points $y_k$ and $y_{k+1}$, and let $H_k$ be the closed half space in $Y$ containing the origin and determined by $l_k$. Let us observe that, for $k\geq 2$, the point  $2y_{k}-y_{k+1}$ belongs to $l_k$ and hence, by Claim~3, it holds
$$d(y_{k-1},l_k)\leq \|2y_{k}-y_{k+1}-y_{k-1}\|\to 0.$$ 

Let us start by putting $\tilde{f}(y)=0$ whenever $h(y)< 0$. If $h(y)\geq 0$ we define  $\tilde{f}$ as follows
$$	\tilde{f}(y):=\begin{cases} \alpha_{2k-1}+\frac{\alpha_{2k+1}-\alpha_{2k-1}}{\alpha_{2k}-\alpha_{2k-1}}\bigl(h(y)-\alpha_{2k-1}\bigr) &\text{if } \alpha_{2k-1}\le h(y)< \alpha_{2k}\\
	\alpha_{2k+1} &\text{if } \alpha_{2k}\le h(y)< \alpha_{2k+1},\ y\in H_{2k}\\
	\alpha_{2k+1} +d(y,H_{2k}) &\text{if } \alpha_{2k}\le h(y)< \alpha_{2k+1},\ y\not\in H_{2k}
\end{cases}$$
It is easy to see that the function $\tilde{f}$ is now well-defined on $D$, it is continuous
and QC, and moreover,
$\tilde{f}$ is $\frac{2}{\beta}$-Lipschitz (since it is ``piecewise $\frac{2}{\beta}$-Lipschitz'').
\bigskip

\noindent{\em Definition of $f\colon C\to\R$, and completion of the proof.} 
\smallskip

\noindent
Now, we simply define $f$ by
$$
f:=\tilde{f}\circ P,
$$
where the projection $P$ onto $Y$ has been defined above.
It remains to show that $f$ cannot be extended to a 
uniformly continuous QC function $F\colon X\to\R$.
Proceeding by contradiction, assume that such $F$ exists. For each $n\ge1$, let $x_n\in\partial C$ be such that
$P(x_n)=y_n$.
For each $k\ge1$, the convex set $[F<\alpha_{2k+1}]_C=[f<\alpha_{2k+1}]_C$
is contained in the half-space $W_{2k}:=P^{-1}(H_{2k})$ whose boundary intersects $C$
in a relatively open subset of the boundary. So, by convexity, $[F<\alpha_{2k+1}]_X$
must be contained in the same half-space. Consequently, by our construction together with
Claim ~3,  we have
\begin{multline}\label{stimarossa}
	d\bigl([F<\alpha_{2k}]_X,[F\ge \alpha_{2k+1}]_X\bigr) 
	\le d\bigl(  x_{2k-1}, X\setminus W_{2k} \bigr)= d\bigl( y_{2k-1}, l_{2k} \bigr) \\ 
	\le \|y_{2k+1}+y_{2k-1}-2 y_{2k}\| \to 0\ \text{as $k\to+\infty$.}
\end{multline}
Since $F$ is uniformly continuous on $X$, it admits an invertible 
modulus of continuity $\omega\colon\R^+\to\R^+$,
that is, an increasing homeomorphism of $\R^+$ such that $|F(x)-F(x')|\le\omega(\|x-x'\|)$ 
for all $x,x'\in X$; see \cite{DEVEqc_UC}.
By \cite[Proposition~2.5]{DEVEqc_UC}, we must have
$$
d\bigl([F<\alpha_{2k}]_X,[F\ge \alpha_{2k+1}]_X\bigr) 
\ge \omega^{-1}(\alpha_{2k+1}-\alpha_{2k}) 
\ge\omega^{-1}(\beta)>0
$$
for each $k$. But this is impossible by \eqref{stimarossa}. 
This contradiction completes the proof of Theorem~\ref{no_ucQC}.
\end{proof}

\medskip

%%%%%%%%%%%%%%%%%%%%%%%%%%%%%%%

\section{Nonexistence of Lipschitz QC extensions}

The aim of this section is to show that for an arbitrary convex body $C$ in an arbitrary normed space $X$ of dimension at least 2, there exists a Lipschitz QC 
function on $C$ which cannot be extended to a
Lipschitz QC function defined on the whole $X$. 
This is a generalization of our previous result from Theorem~\ref{T:old}(a) concerning 
only a particular $C$.

Let us start with two lemmas, the first one quite standard and the second one
elementary but a bit technical.

\begin{lemma}\label{L:immagine}
Let $X,Y$ be arbitrary normed spaces, $E\subset X$ a convex set with nonempty interior,
and $P\colon X\to Y$ a continuous, linear, open mapping. Then
$P(\mathrm{int}_X E)=\mathrm{int}_Y P(E)$.
\end{lemma}

\begin{proof}
Since $P(\mathrm{int}_X E)$ is a nonempty open set, $P(E)$ has nonempty interior
and $P(\mathrm{int}_X E)\subset \mathrm{int}_Y P(E)$. Now, using Fact~\ref{F:convex} and
the continuity of $P$, we obtain
$$
P(E)\subset P\bigl(\overline{E}^X\bigr)=P\bigl(\overline{\mathrm{int}_X E}^X\bigr)\subset\overline{P(\mathrm{int}_X E)}^Y
\subset \overline{\mathrm{int}_Y P(E)}^Y.
$$
Passing to interiors, we have
$$
\mathrm{int}_Y P(E) \subset \mathrm{int}_Y \! \left[\overline{P(\mathrm{int}_X E)}^Y \right]
\subset \mathrm{int_Y}\!\left[ \overline{\mathrm{int}_Y P(E)}^Y \right]\!.
$$
Now, the statement follows since, again by Fact~\ref{F:convex}, the
second set coincides with 
$\mathrm{int}_Y\left[P(\mathrm{int}_X E)\right]=P(\mathrm{int}_X E)$,
and the third set equals $\mathrm{int}_Y\left[\mathrm{int}_Y P(E)\right]=\mathrm{int}_Y P(E)$.
\end{proof}

\medskip

Let $M$ be a nonempty set in a normed space. Recall that the {\em minimal modulus
of continuity} of a function $g\colon M\to\R$  is the non-decreasing function
$$
\omega_g\colon[0,+\infty)\to[0,+\infty],\ 
\omega_g(t):=\sup\{|g(x)-g(y)|:\; x,y\in M,\;\|x-y\|\le t\}.
$$
It is clear that $\omega_g(0)=0$, and $\omega_g$ is right-continuous at $0$ if and only if
$g$ is uniformly continuous. It is elementary to see that $g$ is $K$-Lipschitz if and only if $\omega_g(t)\le Kt$ for each $t\ge0$.

\medskip

\begin{lemma}\label{lemma: quoziente dimensione 2}
Let $X$ be an arbitrary normed space of dimension at least $2$, and let $Y=\R^2$ (with the euclidean norm). Let 
$P\colon X\to Y$ be a bounded linear operator such that
$\theta B_Y\subset P(B_X)$ for some $\theta>0$. Let $\alpha<\beta$ be real numbers.  
Assume that:
	\begin{enumerate}
		\item $E\subset X$ is a convex body such that $C:=\overline{P(E)}\ne Y$ (and hence $C$ is a convex body in $Y$);
		\item $f\colon C\to\R$ is a continuous QC function;
		\item the sets $C_\alpha:=[f\leq\alpha]_C$ and $C_\beta:=[f\leq\beta]_C$ are convex bodies;
		\item there exist  a line $L\subset Y$ and an open convex set $U\subset C$ 
such that $U\cap\partial_Y C_\beta= U\cap L\ne\emptyset$;
		\item $F\colon E\to\R$ is the continuous QC function defined by $F(x)=f\bigl(P(x)\bigr)$.		 
	\end{enumerate}
Let $\tilde{F}\colon X\to \R$ be a QC extension of $F$, and let $\omega_{\tilde F}$ be its minimal modulus of continuity.
Then: 
\begin{enumerate}[(a)]
	\item\label{A} $P\bigl([\tilde F\leq \beta]_X\bigr)\subset H$ where $H\subset Y$ is the closed half-plane determined by $L$ and containing $C_\beta$;
	\item\label{B} $\omega_{\tilde F}\left(\frac{d(L, C_\alpha)+\eta}{\theta}\right)>\beta-\alpha$ for each $\eta>0$.
\end{enumerate}
\end{lemma}

\begin{proof}
Since $P$ is an open mapping, $C$ is a convex body in $Y$ and $\mathrm{int}_Y C=\mathrm{int}_Y P(E) =  P(\mathrm{int}_X E)$
(by Fact~\ref{F:convex} and Lemma~\ref{L:immagine}).  
The convex set 
$\tilde{C}_\beta:=P([\tilde F\leq \beta]_X)$ contains $C_\beta$, and $\tilde{C}_\beta\cap U=C_\beta\cap U$. 
Since the line $L$ supports $C_\beta$ at the points of $U$,
an easy convexity argument shows that
$L$ supports $\tilde{C}_\beta$ as well. This proves \eqref{A}.
	
	Let us prove \eqref{B}. Given $\eta>0$,
	let $y_0\in \mathrm{int}_Y C_\alpha$ and $z_0\in Y\setminus H$
be such that $\|z_0-y_0\|= d(L,C_\alpha)+\eta$.
	Let $v_0\in \mathrm{int}_X E$ be such that $y_0=P(v_0)$, and notice that $\tilde F(v_0)=f(y_0)\leq \alpha$. Observe that, since $\theta B_Y\subset P(B_X)$, there exists $w_0\in X$ such that $\|w_0-v_0\|\leq \frac{d(L,C_\alpha)+\eta}{\theta}$ and $P(w_0)=z_0$. By \eqref{A} and since $z_0\in X\setminus H$, we have $\tilde F(w_0)>\beta$. Hence,
	$$\omega_{\tilde F}\left(\frac{d(L, C_\alpha)+\eta}{\theta}\right)\geq \tilde F(w_0)-\tilde F(v_0)>\beta-\alpha.$$
\end{proof}

\begin{theorem}\label{no_lipQC}
	Let $X$ be an arbitrary normed space of dimension at least 2,  and let $E\subset X$ be a convex body.
	Then there exists a Lipschitz QC function $F\colon E\to\R$ that does not admit any 
Lipschitz QC extension to the whole $X$.
\end{theorem}
 
\begin{proof}
	Let $x^*\in X^*\setminus\{0\}$ be any supporting functional of $E$ and let $Z$ be a closed 
subspace of codimension $1$ in $\ker(x^*)$. 
Let  $Y$ be the quotient $X/Z$ and let $P\colon X\to Y$ be the corresponding quotient map. 
Since $Y$ is $2$-dimensional, it can be identified with the euclidean $\R^2$.
Consider the convex set $C:=\overline{P(E)}$ 
which is clearly different from the whole $Y=\R^2$. So $C$ is a convex body in $Y$. 	
Without any loss of generality, we can suppose that:
$$
(0,0)\in\partial C\,,\quad
C\subset \{(u,v)\in\R^2: v\geq 0\}\,,\quad
(0,1)\in\inte C\,.
$$
Let $\theta>0$ be such that $\theta B_Y\subset P(B_X)$.
Notice that there exists $\epsilon\in(0,1)$ such that the set $\partial C\cap\bigl([0,\epsilon]\times[0,1]\bigr)$ 
coincides with the graph of  a convex continuous non-decreasing function $g\colon [0,\epsilon]\to [0,1]$ with $g(0)=0$. 
For $z\in[0,\epsilon]$, notice that  $\delta(z):=\frac{g(z)}{2} -g(\frac{z}{2})\equiv \frac{g(0)+g(z)}{2} -g(\frac{0+z}{2})\ge 0$.
Since  $g(z)=2\delta(z)+2 g(\frac{z}2)$, an easy inductive argument gives that for each integer $n\ge0$,
$$
g(\epsilon)=2\sum_{k=0}^n2^k\delta\left(\frac{\epsilon}{2^k}\right)+2^{n+1}g\left(\frac{\epsilon}{2^{n+1}}\right)
\geq 2\sum_{k=0}^n2^k\delta\left(\frac{\epsilon}{2^k}\right).
$$
Thus the series $\sum_{k=0}^{\infty}2^k\delta\left(\frac{\epsilon}{2^k}\right)$ converges and hence 
$2^k\delta\left(\frac{\epsilon}{2^k}\right)\to 0$ as $k\to\infty$.
Now, let $\{\alpha_k\}_{k\in\N\cup\{0\}}\subset(0,1)$ be a sequence such that ${2^k}{\alpha_k}\to0$, and define
$$\textstyle D_k=\left\{(u,v)\in C:\; u\geq \frac34\frac{\epsilon}{2^k},\; v\geq \alpha_k +\frac{2^k}{\epsilon}\left(g\left(\frac{\epsilon}{2^k}\right)-\alpha_k\right)u\right\},
\quad \text{for $k\in\N\cup\{0\}$.}
$$
Notice that:
\begin{itemize}
	\item $\left\{(u,v)\in C:\; u\geq \frac{\epsilon}{2^k} \right\}\subset D_k\subset\left\{(u,v)\in C:\; u\geq \frac34\frac{\epsilon}{2^k} \right\}$;
	\item $D_k\subset D_{k+1}$ and $\bigcup_{k\in\N} D_k=\left\{(u,v)\in C: u>0 \right\}$;
	\item  if we consider the line 
$$\textstyle l_k=\left\{(u,v)\in\R^2:\;  v= \alpha_k +\frac{2^k}{\epsilon}\left(g\left(\frac{\epsilon}{2^k}\right)-\alpha_k\right)u\right\},$$
	and the points $P_k:=\left( \frac{\epsilon}{2^{k}},g\left(\frac{\epsilon}{2^{k}}\right)\right)$ and
	$$\textstyle 
Q_k:=\left( \frac{\epsilon}{2^{k}},2g\left(\frac{\epsilon}{2^{k+1}}\right)-\alpha_{k+1}\right)=
\left( \frac{\epsilon}{2^{k}},g\left(\frac{\epsilon}{2^{k}}\right)-2\delta\left(\frac{\epsilon}{2^{k}}\right)-\alpha_{k+1}\right),$$ 
	then $P_k\in D_k$, $Q_k\in l_{k+1}$ and
\begin{equation}\label{eq: distanza o-piccolo}
\textstyle d\bigl(l_{k+1}, D_k\bigr)\leq d\bigl(Q_k, P_k\bigr)= 2\delta\left(\frac{\epsilon}{2^{k}}\right)+\alpha_{k+1};
	\end{equation}
	\item $d\bigl(C\setminus D_{k+1}, D_k\bigr)\geq \frac34 \frac\epsilon{2^k}-\frac\epsilon{2^{k+1}} =\frac{\epsilon}{2^{k+2}}$.
\end{itemize}
Let $f\colon C\to\R$ be defined by: 
$$\textstyle{f}(x)=\begin{cases} 0 &\text{if $x\in D_0$}, \\
	\frac\epsilon4\left(2-\frac{1}{2^{k-1}}\right)+d(x,D_k) &\text{if $k\ge0$, $ x\not\in D_k$, $d(x,D_k)\leq \frac{\epsilon}{2^{k+2}}$,} \\
	\frac\epsilon4\left(2-\frac{1}{2^{k}}\right) &\text{if $k\ge0$, $x\in D_{k+1}$, $d(x,D_k)> \frac{\epsilon}{2^{k+2}}$,} \\
	\frac\epsilon2 &\text{if $u\leq 0$.}  
\end{cases}$$
It is not difficult to see that $f$ is a quasiconvex 1-Lipschitz function such that 
for $\beta_k=\frac\epsilon4\left(2-\frac{1}{2^{k-1}}\right)$ ($k\ge0$)
one has
$$\left[{f}\leq\beta_k \right]_C=D_{k}.$$
Now, let $F\colon E\to\R$ be the Lipschitz quasiconvex function defined by $F(x)=f\bigl(P(x)\bigr)$.
It remains to show that $F$ cannot be extended to a 
Lipschitz quasiconvex function $\tilde F\colon X\to\R$.
Proceeding by contradiction, assume that such $\tilde F$ exists and let $K>0$ be such that $\omega_{\tilde F}(t)\leq Kt$ whenever $t\geq0$. 
Fix $k\in\N$. An application of Lemma~\ref{lemma: quoziente dimensione 2} with
$$L=l_{k+1},\quad \alpha=\beta_k,\quad \beta=\beta_{k+1},\quad \eta=d\bigl(l_{k+1}, D_k\bigr)$$
yields 
$$\textstyle \omega_{\tilde F}\left(\frac{2d\left(l_{k+1}, D_k\right)}{\theta}\right)\geq \frac{\epsilon}{2^{k+2}}.$$
By \eqref{eq: distanza o-piccolo} we have
$$\textstyle 2K\frac{2\delta\left(\frac{\epsilon}{2^{k}}\right)+\alpha_{k+1}}{\theta}\geq\omega_{\tilde F}\left(2\frac{2\delta\left(\frac{\epsilon}{2^{k}}\right)+\alpha_{k+1}}{\theta}\right)\geq \frac{\epsilon}{2^{k+2}},$$
and hence 
$$\textstyle 2^k\left(2\delta\left(\frac{\epsilon}{2^{k}}\right)+\alpha_{k+1}\right)\geq\ \frac{\theta\epsilon}{8K}.$$
But this is a contradiction since the left-hand term tends to $0$ as $k\to\infty$.
\end{proof}

\begin{figure}
	\centering

	\begin{tikzpicture}[scale=1]
		\draw[->](-3,0)--(9,0)node[below]{\small{\( u \)}};
		\draw[->](0,-1)--(0,6.5)node[right]{\small{\( v \)}};

	\draw[convexset]
(-3,6) .. controls + (0.1,-0.5) and + (-.3,0.5) ..(-2,1) .. controls + (0.3,-0.5) and + (-0.5,0) .. (-1,0)  .. controls + (.5,0) and + (-0.4,0) ..(0,0) .. controls + (1.5,0) and + (-1,-0.4) .. (4,1) .. controls + (1,0.4) and + (-.5,-1) ..(8,5).. controls + (.5,1) and + (-.5,-1) ..(8.5,6);

\draw[line width = .7 pt]
(-3,6) .. controls + (0.1,-0.5) and + (-.3,0.5) ..(-2,1) .. controls + (0.3,-0.5) and + (-0.5,0) .. (-1,0)  .. controls + (.5,0) and + (-0.4,0) ..(0,0) .. controls + (1.5,0) and + (-1,-0.4) .. (4,1) .. controls + (1,0.4) and + (-.5,-1) ..(8,5).. controls + (.5,1) and + (-.5,-1) ..(8.5,6);

	\draw (8,0) circle[radius=1pt];
\fill (8,0) circle[radius=1pt];

\draw (4,0) circle[radius=1pt];
\fill (4,0) circle[radius=1pt];

\draw (8,5) circle[radius=1pt];
\fill (8,5) circle[radius=1pt];
	\node at (8.3,4.7) {\small{$P_k$}};

\node at (7,5.5) {{$D_k$}};

\draw (0,.5) circle[radius=1pt];
\fill (0,.5) circle[radius=1pt];
	\node at (-.3,.5) {\small{$\alpha_k$}};

\draw (0,.2) circle[radius=1pt];
\fill (0,.2) circle[radius=1pt];
	\node at (-.5,.2) {\small{$\alpha_{k+1}$}};

	\draw[dashed,thick, line width = 0.2 pt](0,.5)--(6,3.875);
	\draw[dashed,thick, line width = 0.2 pt](6,0)--(6,3.875);
	\draw[dashed,thick, line width = 0.2 pt](3,0)--(3,.8);
	\draw[dashed,thick, line width = 0.2 pt](4,0)--(4,1);
	\draw[dashed,thick, line width = 0.2 pt](8,0)--(8,5);
\draw[-,thick, line width = 0.7 pt](6,3.875)--(8,5);

\draw[-,thick, line width = 0.7 pt](6,3.875)--(6,6);

\draw[dashed,thick, line width = 0.2 pt](0,.2)--(3,.8);

\draw[-,thick, line width = 0.7 pt](3,.8)--(4,1);

\draw[-,thick, line width = 0.7 pt](3,.8)--(3,6);
\draw[dashed,thick, line width = 0.2 pt](4,1)--(9,2);

\draw (8,1.8) circle[radius=1pt];
\fill (8,1.8) circle[radius=1pt];
\node at (8.3,1.5) {\small{$Q_{k}$}};
\node at (9.5,2) {\small{$l_{k+1}$}};

\node at (4,4) {{$D_{k+1}$}};

\draw (4,1) circle[radius=1pt];
\fill (4,1) circle[radius=1pt];

	\node at (4.5,.7) {\small{$P_{k+1}$}};

\node at (4,-.4) {\small{$\frac{\epsilon}{2^{k+1}}$}};

\node at (8,-.4) {\small{$\frac{\epsilon}{2^{k}}$}};

\node at (6,-.4) {\small{$\frac34\frac{\epsilon}{2^{k}}$}};

\draw (6,0) circle[radius=1pt];
\fill (6,0) circle[radius=1pt];

\draw (3,0) circle[radius=1pt];
\fill (3,0) circle[radius=1pt];
\node at (3,-.4) {\small{$\frac34\frac{\epsilon}{2^{k+1}}$}};

	\end{tikzpicture}

\caption{Construction of the sets $D_k$}\label{FIG-Dk}
\end{figure}
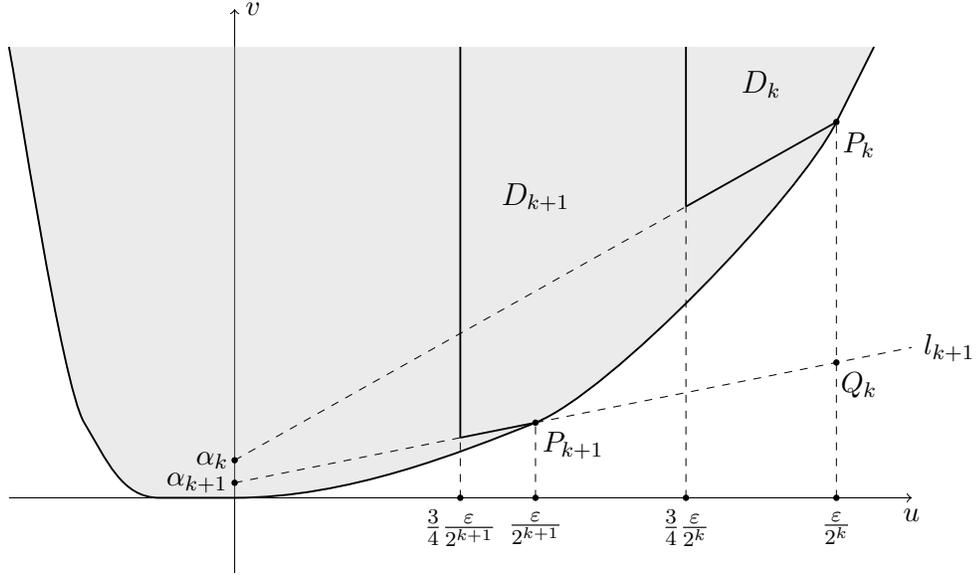

\medskip

%%%%%%%%%%%%%%%%%%%%%%%%%%%%%%%

\section{Convex bodies with no asymptotic directions}

The present section contains results about convex bodies without asymp\-to\-tic directions
that will be needed for our positive results concerning extendability of QC functions. 
Let us remark that closed convex sets without
asymptotic directions were considered and studied by D.~Gale and V.~Klee
\cite{GaleKlee} under the name ``continuous convex bodies''. Their terminology
comes from the fact that such convex sets $K\subset X$ are characterized by the property that
their support function $h\mapsto \sup h(K)$ is continuous as a mapping from
the dual unit sphere of $X$ into $\R\cup\{+\infty\}$. Moreover,
the class of such sets is closed under finite algebraic sums and under
convex hulls of finite unions.

Let us start with an elementary lemma. Let us recall that, given an open interval $I\subset\R$ and 
a convex function $\phi\colon I\to\R$, the left 
and right derivatives $\phi'_-$ and $\phi'_+$ exist on $I$ and satisfy $\phi'_-(t_1)\leq\phi'_+(t_1)\leq\phi'_-(t_2)$, whenever $t_1,t_2\in I$ and $t_1<t_2$.

\begin{lemma}\label{L:cv}
Let $-\infty\le a<+\infty$, and let $\phi\colon(a,+\infty)\to\R$ be a convex function.
Then
\begin{equation}\label{E:cv}
\lim_{t\to+\infty}\frac{\phi(t)}{t}=\lim_{t\to+\infty}\phi'(t)\,,
\end{equation}
where $\phi'(t)$ denotes an arbitrary element of $\partial\phi(t)\equiv[\phi'_-(t),\phi'_+(t)]$.
\end{lemma}

\begin{proof}
By convexity of $\phi$, $\phi'$ is non-decreasing and hence the second limit in \eqref{E:cv} exists and 
equals $s:=\sup_{t>a}\phi'(t)$. Also the first limit in \eqref{E:cv} exists. To see this,
fix $t_0>a$. The function $t\mapsto\frac{\phi(t)-\phi(t_0)}{t-t_0}$ is non-decreasing,
and for any $t_1>t_0$ its values in $(t_0,t_1)$ 
lie in $[\phi'(t_0),\phi'(t_1)]$. Moreover, 
$$
\frac{\phi(t)}{t}=\frac{\phi(t)-\phi(t_0)}{t-t_0}\cdot\frac{t-t_0}{t}+\frac{\phi(t_0)}{t}.
$$
Hence the first limit in \eqref{E:cv} exists, and
$$
\phi'(t_0)\le\lim_{t\to+\infty}\frac{\phi(t)}{t}\le\lim_{t\to+\infty}\phi'(t)=s.
$$
Since this holds for every $t_0>a$, the rest  follows easily.
\end{proof}

\begin{corollary}\label{C:adir}
Let $C\subset X$ be an unbounded convex body, $x_0\in X\setminus(\mathrm{int}\,C)$, 
and let $v\ne0$ be
such that $v$ is not an asymptotic direction for $C$ and 
$(x_0+\R_+v)\cap(\mathrm{int}\,C)=\emptyset$. Then
$$
\lim_{t\to+\infty}\frac{d(x_0+tv,C)}{t}>0\,.
$$
\end{corollary}

\begin{proof}
The function $\phi(t):=d(x_0+tv,C)$ is continuous and convex on $\R_+$
and hence there exists $\ell:=\lim_{t\to+\infty}\phi(t)\in(0,+\infty]$.
(Notice that $\ell$ cannot be zero since $v$ is not an asymptotic direction
for $C$.) Denote $\delta:=\inf\phi(\R_+)=\mathrm{dist}(x_0+\R_+v, C)\le\ell$.
We claim that $\delta<\ell$.

To show this, assume the contrary. There exists a sequence $t_n\to+\infty$
such that $\phi(t_n)\to\ell$. In other words, denoting $y_n=x_0+t_nv$,
there exist $c_n\in C$ ($n\in\N$) with $\|c_n-y_n\|\to\ell=\delta\in(0,+\infty)$. 
By compactness, we can (and do) assume that $c_n-y_n\to w$. Let $x_1=x_0+w$.
Since $\|w\|=\delta$, $(x_1+\R_+v)\cap(\mathrm{int}\,C)=\emptyset$.
Moreover, $\|x_1+t_nv-c_n\|=\|w-(c_n-y_n)\|\to0$. But this is a contradiction
since $v$ is not an asymptotic direction for $C$.

Since $\delta<\ell$,
$\phi$
must be strictly increasing in a neighborhood of $+\infty$. By
Lemma~\ref{L:cv},  
$\lim_{t\to+\infty}\frac{\phi(t)}{t}=\sup_{t>0}\phi'(t)>0$.
We are done.
\end{proof}

\begin{lemma}\label{L:bdd_sublevels}
Let $\mathrm{dim}\,X=d\ge2$, and let $C\subset X$ be an unbounded convex body 
with no asymptotic directions. Let a nonzero functional $h\in X^*$ be bounded below on $C$.
Then:
\begin{enumerate}[(a)]
\item $h$ attains its infimum over $C$;
\item each sublevel set $C_\alpha:=[h\le\alpha]_C$ ($\alpha\in\R$) is bounded.
\end{enumerate}
\end{lemma}

\begin{proof}
(a) Denote $m:=\inf h(C)$ and $L:=h^{-1}(m)$. Proceeding by contradiction,
assume that $C\subset [h>m]_X$. So we have $C\cap L=\emptyset$ and $d(C,L)=0$.
Fix some $x_0\in L$. There exist sequences $\{v_n\}_n\in h^{-1}(0)\cap S_X$ 
and $\{t_n\}_n\subset(0,+\infty)$
such that $d(x_0+t_n v_n,C)\to 0$. Now, we must have $t_n\to+\infty$, since otherwise
an easy compactness argument would give a common point of $C$ and $L$. By
passing to a subsequence, we can (and do) assume that $v_n\to v\in h^{-1}(0)\cap S_X$.
Since $v$ is not an asymptotic direction, we must have
$(1/t_n)\,d(x_0+t_n v,C)\to\beta>0$. But then we get
$$
\frac{d(x_0+t_n v_n,C)}{t_n}\ge \frac{d(x_0+t_n v,C)}{t_n} - \|v_n-v\| \to\beta
$$
which implies that $d(x_0+t_n v_n,C)\to+\infty$, and this is a contradiction.

(b) We proceed as in the first part of the proof of Theorem~\ref{no_ucQC}. If (b) is false, there exists $\alpha>m$ with $C_\alpha$ unbounded.
Then $C_\alpha$ contains a half-line of the type $c+\R^+ v$.
Since $\R^+ h(v)=h(c+\R^+v)-h(c)\subset[m-\alpha,\alpha-m]$, we must have
$h(v)=0$. But then $C_m$ contains a half-line of the type $d+\R^+v$,
but this is a contradiction since $C_m\subset\partial C$ and $v$ is not an asymptotic 
direction for $C$.
\end{proof}

\begin{notation}
Given a set $E\subset X$, we define $\mathrm{cone}(E)=\R_+E$ (the cone
generated by $E$). For $x\in X$, we denote 
$$
c(x,E)=x+\mathrm{cone}(E -x),
$$
the cone with vertex $x$, generated by $E$. Notice that $c(0,E)=\mathrm{cone}(E)$.
\end{notation}

\begin{lemma}\label{L:Gamma}
Let $\mathrm{dim}\,X=d\ge2$ and let $C\subset X$ be an unbounded 
convex body with no asymptotic directions.
Then for each $z\in X\setminus C$, the cone $c(z,C)$ is closed and the set
$$
\Gamma_C(z)=C\cap\partial[c(z,C)]
$$
is nonempty and
bounded.
\end{lemma}

\begin{proof}
(a) By translation, we can (and do) assume that $z=0\notin C$. 
Suppose that
there exists a norm-one vector $v\in\partial[\mathrm{cone}(C)] \setminus \mathrm{cone}(C)$.
Then $\R_+v\cap C=\emptyset$ and there exists a sequence $\{c_n\}_n\subset C$
with $\delta_n:=\left\| \frac{c_n}{\|c_n\|} - v \right\| \to 0$. But then
$0=\lim_{n\to\infty}\delta_n=\lim_{n\to\infty}\frac{\bigl\|c_n-\|c_n\|v\bigr\|}{\|c_n\|}\ge
\lim_{n\to\infty}\frac{d(\|c_n\|v,C)}{\|c_n\|}$. Since $v$ is not an asymptotic direction for $C$,
Corollary~\ref{C:adir} implies that $\{c_n\}_n$ is bounded. Moreover, since $0\notin C$, 
we have $\inf_n \|c_n\|>0$. Passing to a subsequence, we may (and do) assume that
$\|c_n\|\to\tau>0$. But this implies that $c_n\to\tau v$ and hence $\tau v\in C$,
so that $v\in\mathrm{cone}(C)$. This contradiction proves that  $c(z,C)$ is closed.

(b) The set $\partial[c(z,C)]$ is contained in $c(z,C)$, and hence each its
half-line from $z$ intersects $C$. Thus $\Gamma_C(z)$ is nonempty.
To prove that it is also bounded, assume the contrary, that is,
 there exists a sequence $\{c_n\}_{n}\subset\Gamma_C(z)$
such that $\|c_n\|\to+\infty$. Denote $u_n:=\frac{c_n-z}{\|c_n-z\|}$, and notice that
$\|c_n-z\|\to+\infty$. Passing to a subsequence, we can (and do) assume that
$u_n\to\bar{u}\in S_X$. 
We claim that $({z}+\R_+\bar{u})\cap(\mathrm{int}\,C)=\emptyset$. (Indeed, if not,
there exists $\tau>0$ such that ${z}+\tau\bar{u}\in\mathrm{int}\,C$. Then,
for each sufficiently large $n$, $z+\tau u_n\in\mathrm{int}\,C$. Hence 
$c_n\in c(z,\mathrm{int}\,C)=\mathrm{int}\,c(z,C)$, and this is a contradiction).
Since $\bar{u}$ is not an asymptotic direction, we have
$\lim_{t\to+\infty}\frac{d({z}+t\bar{u},C)}{t}>0$
(Corollary~\ref{C:adir}). On the other hand,
\begin{align*}
\frac{d\bigl({z}+\|c_n-z\|\bar{u},C\bigr)}{\|c_n-z\|} &\le
\frac{\bigl\| ({z}+\|c_n-z\|\bar{u}) - (z+\|c_n-z\|u_n)  \bigr\|}{\|c_n-z\|}
\\ &= \|\bar{u}-u_n\| \to0.
\end{align*}
This contradiction completes the proof.
\end{proof}

\begin{definition}
Given a convex body $C\subset X$ and $x\in\partial C$, let us define:
\begin{itemize}
\item $\Sigma_1(x,C)=\{g\in S_{X^*}: \sup g(C)=g(x)\}$ (the set of all norm-one supporting fucntionals to $C$ at $x$);
\item $K(x,C)=\bigcap_{g\in\Sigma_1(x,C)}[g\le g(x)]_X$ (the intersection of all half-spaces containing $C$ whose 
boundary is a supporting hyperplane to $C$ at $x$).
\end{itemize}
\end{definition}

Using a simple application of the Hahn-Banach separation theorem,
it is easy to see that $K(x,C)=\overline{c(x,C)}$. However, we shall not need this fact.

\begin{lemma}\label{L:good}
Let $\mathrm{dim}\,X=d\ge2$. Let $C\subset X$ be a convex body, and let
$x\in X\setminus C$ be such that the cone $c(x,C)$ is closed. Denote
$\Gamma_C(x):=\partial[c(x,C)]\cap C$. Then
$$
x\in K(y,C) \quad \text{for each $y\in\partial C\setminus\mathrm{conv}[\Gamma_C(x)\cup\{x\}]$.}
$$
\end{lemma}

\begin{proof}
By translation, we can (and do) assume that $x=0\notin C$. 
So we have that $\mathrm{cone}(C)$ is closed
and we consider $\Gamma_C(0)=\partial[\mathrm{cone}(C)]\cap C$.
Let $y\in \partial C\setminus \mathrm{conv}[\Gamma_C(0)\cup\{0\}]$ and $g\in\Sigma_1(y,C)$,
and denote $s:=\sup g(C)=g(y)$. Since $y\notin \Gamma_C(0)$, we have 
$y\in\mathrm{int}[\mathrm{cone}(C)]=\mathrm{cone}(\mathrm{int}\,C)$.
Thus there exists $t>0$ such that $ty\in\mathrm{int}\,C$. Clearly, $t\ne 1$.

Assume that $t>1$. Since $g(ty)<g(y)$, we must have $s\equiv g(y)<0$. Take an arbitrary
two-dimensional subspace $Y\subset  X$ that contains $y$. The boundary
$\partial_Y[\mathrm{cone}(C\cap Y)]=\partial_Y[\mathrm{cone}(C)\cap Y]$ consists
of two half-lines, say $\R^+u$ and $\R^+v$ with $u,v\in C$. 
Notice that the two half-lines
cannot coincide (since $Y$ intersects $\mathrm{int}\,C$) and cannot
be opposite to each other (since $0\notin C$).
Clearly, $u,v\in\Gamma_C(0)$.
The segment $[u,v]$ must contain a positive multiple of $y$, that is, there exist $\tau>0$
and $\lambda\in(0,1)$ such that $\tau y=(1-\lambda)u+\lambda v$. Applying $g$,
we obtain $\tau s=(1-\lambda)g(u)+\lambda g(v)\le s$, and hence $\tau\ge1$. 
Consequently, $y=\frac{1-\lambda}{\tau}u + \frac{\lambda}{\tau}v + \left(1-\frac1\tau\right)0
\in\mathrm{conv}[\Gamma_C(0)\cup\{0\}]$. But this contradicts our assumptions.

So we must have $t\in(0,1)$. Since $g(ty)<g(y)$, we have $g(y)>0$. Thus
$0\in[g\le g(y)]_X$. Since $g\in \Sigma_1(y,C)$ was arbitrary, we conclude that
$0\in K(y,C)$ as needed.
\end{proof}

\medskip

\subsection*{An extension method}
\ 
\\
Now, we are going to asign to any convex body $B$, contained in a given fixed convex body $C$, its 
``extension'' $e(B)\subset X$; and we also define $e(\emptyset)$.
This extension method is practically the same as in our previous paper \cite{DEVEqc_UC}. The only
substantial difference is that we consider here closed convex sets instead of open ones. We also
use a different notation.

\begin{definition}\label{D:e}
Let $C$ be a convex body in a normed space $X$, and let $B\subset C$ be a closed convex set such that
$B=\overline{\mathrm{int}\,B}$ (equivalently, $B$ is either empty or a convex body).
If  $\emptyset\ne B\ne C$, we define
\begin{equation}\label{E:e}
e(B)=
\bigcap_{y\in\overline{(\mathrm{int}\,C)\cap \partial B}}K(y,B) = \bigcap_{y\in\partial_C B}K(y,B).
\end{equation}
We also define $e(\emptyset)=\emptyset$ and $e(C)=X$.
\end{definition}

\begin{remark}\label{R:e}
Concerning \eqref{E:e}, the equality $\overline{(\mathrm{int}\,C)\cap \partial B}=\partial_C B$ is an easy exercise.
Notice that $e(B)$ is a closed convex set.
Moreover, an elementary Hahn-Banach separation argument implies that $e(B)\cap C=B$.
\end{remark}

\begin{proposition}\label{P:riempimento}
Let $\mathrm{dim}\,X=d\ge2$, and let $C\subset X$ be a (bounded or unbounded) convex body 
with no asymptotic directions. Let $\{B_n\}_n$ be an increasing 
sequence of convex bodies contained in $C$
such that $C=\bigcup_n \mathrm{int}_C B_n$. 
Then $X=\bigcup_n e(B_n)$.
\end{proposition}
\begin{proof}
Notice that $C=\bigcup_n B_n\subset \bigcup_n e(B_n)$. 
Fix an arbitrary $x\in X\setminus C$. Let $\Gamma_C$ be as in
Lemma~\ref{L:Gamma}. The same lemma implies that 
$K:=\mathrm{conv}[\Gamma_C(x)\cup\{x\}] \cap C$ 
is a compact subset of $C$. By our assumption, there exists an index $m$
for which $K\subset \mathrm{int}_C B_m$. Consequently, $\partial_C B_m\cap K=\emptyset$.
Moreover,  since $\Gamma_C(x)$ is contained in $\mathrm{int}_C B_m$, we must have 
$\Gamma_{B_m}(x)=\Gamma_C(x)$.
Since $\partial_C B_m$ does not intersect $\mathrm{conv}[\Gamma_{B_m}(x)\cup\{x\}]$, 
we can apply Lemma~\ref{L:good}
 to conclude 
that $x\in\bigcap_{y\in\partial_C B_m}K(y,B_m)=e(B_m)$. We are done.
\end{proof}
\begin{observation}\label{O:riempimento}
Let $C$ be a convex body in a finite-dimensional normed space $X$, and let 
$\{B_n\}_n$ be an increasing sequence of convex bodies contained in $C$. Then
$C=\bigcup_n \mathrm{int}_C B_n$ if and only if
\begin{enumerate}
\item[$(*)$] for each $r>0$ there exists $m\in\N$ such that $C\cap r B_X\subset B_m$.
\end{enumerate}
{\em
(Indeed, the necessity follows by compactness; for sufficiency, if
$x\in C$ it suffices to consider some $r>\|x\|$.)
}
\end{observation}

\begin{lemma}\label{L:intersezione}
Let $\mathrm{dim}\,X=d\ge2$, and let $C\subset X$ be a convex body 
with no asymptotic directions. Let $\{B_n\}_n$ be a sequence of convex bodies contained in $C$
such that $B_{n+1}\subset B_n$ for each $n$, and $\bigcap_n B_n=\emptyset$.
Then $\bigcap_n e(B_n)=\emptyset$.
\end{lemma}

\begin{proof}
Notice that the assumptions imply that $C$ must be unbounded.
Fix $h\in S_{X^*}$ such that $m:=\inf h(C)>-\infty$. Since each sublevel set $[h\le\alpha]_C$ is compact
(Lemma~\ref{L:bdd_sublevels}), $[h\le\alpha]_C$ can intersect only finitely many $B_n$'s. Consequently,
$\alpha_n:=\inf h(B_n)\to+\infty$. We can assume that $\alpha_1>m$ (and hence $\alpha_n>m$ for each $n$).
By compactness, moreover, for each $n$ there is $y_n\in B_n$ such that $h(y_n)=\alpha_n$.
It follows that $y_n\in\partial_C B_n$ and $-h\in\Sigma_1(B_n,y_n)$. Consequently,
$e(B_n)\subset[h\ge\alpha_n]_X$ for each $n$, and this implies that $\bigcap_n e(B_n)=\emptyset$.
\end{proof}

The following lemma is an easy extension of \cite[Lemma~5.3]{DEVEqc_UC} to unbounded bodies.

\begin{lemma}\label{L:5.3.bis}
Let $C$ be a (bounded or unbounded) convex body in an arbitrary normed space, and let $C$ 
contain no line.
Let $B\subset C$ be a convex body, $x\in e(B)\setminus B\equiv e(B)\setminus C$ and $y\in C\setminus B$.
Then $[x,y]\cap B\ne\emptyset$.
\end{lemma}

\begin{proof}
If $y\in\mathrm{int}\,C$, let $y'$ be the unique point of the intersection $[x,y]\cap \partial C$.
If $y'\in B$, we are done; if not we can consider $y'$ instead of $y$. Thus we can (and do)
assume that $y\in\partial C$. By translation, we can (and do) assume that
$0\in\mathrm{int}\,B \setminus\mathrm{aff}\{x,y\}$.
Notice that the subspace $Y:=\mathrm{span}\{x,y\}$ is two-dimensional. 
Let $x'$ be the unique element of $[0,x]\cap\partial B$. By \cite[Lemma~5.2]{DEVEqc_UC},
$x'\in\mathrm{int}_{\partial C}(\partial B \cap \partial C)$. Consider the convex body
$D:= C\cap Y$ in $Y$. If $D$ is bounded, then $\partial_Y D$ is a closed simple arc and we can
repeat the argument from the proof
of \cite[Lemma~5.3]{DEVEqc_UC}. The proof of the case of $D$ unbounded
will be done in a quite similar way.

Let $D$ be unbounded. Since $D$ contains a half-line but not any line,
$\partial D$ is homeomorphic to the real line.
Let $\gamma$ be the arc of $\partial D$ with endpoints $x'$ and $y$.
Notice that we can write $\partial C$ as the union of three disjoint sets 
$$
\partial C\; = \;(\partial C\setminus B)\; \cup\;\mathrm{int}_{\partial C}(\partial B\cap\partial C)
\;\cup\;\partial_{\partial C}(\partial B\cap\partial C),
$$
where the first one and the second one are relatively open in $\partial C$ and
they contain
$y$ and $x'$, respectively. Since $\gamma$ is connected,
there must exist $u\in\gamma\cap\partial_{\partial C}(\partial B\cap\partial C)$.
Let $u'$ be the unique common point of $[x,y]$ and the half-line $\R_+u$.
There exist $t>0$ and $\lambda\in(0,1)$ such that 
$tu=u'=(1-\lambda)x+\lambda y$.
By \cite[Lemma~5.2(d)]{DEVEqc_UC}, 
$u \in \overline{\mathrm{int}\,C \cap \partial B}\equiv \partial_C B$.
Consider $g\in \Sigma_1(u,C)$, and observe that $g(u)>0$, and
$g(x)\le g(u)$ since $x\in e(B)$ and $g\in\Sigma_1(u,B)$. 
Then $t g(u)=(1-\lambda)g(x)+\lambda g(y)\le g(u)$,
and hence $t\le 1$. Consequently, $u'\in[0,u]\cap[x,y]\subset B\cap[x,y]$.
\end{proof}

\medskip

\section{Extension from unbounded rotund bodies}

The main result of the present section will be Theorem~\ref{T:ext_contQC}. Before stating it, we need some
preparation.

Let $C$ be a convex body in a normed space $X$. Recall that $C$ is said to be {\em rotund}
if $\partial C$ does not 
contain any nontrivial line segment.

Now, let us define the {\em modulus of local uniform rotundity}
of $C$ at a point $x\in\partial C$. 
Let $r>0$ be the radius of a closed ball
contained in $C$. For
$\epsilon\in[0,2r)$ we define
\begin{equation}\label{E:dc}
\begin{aligned}
\delta_C(x,\epsilon)&\textstyle :=\inf \bigl\{  
d(\frac{x+y}2, \partial C):\; y\in\partial C,\; \|x-y\|\ge\epsilon
\bigr\} \\
&\textstyle =\inf \bigl\{  
d(\frac{x+y}2, \partial C):\; y\in\partial C,\; \|x-y\|=\epsilon
\bigr\}  \\
&\textstyle =\inf \bigl\{  
d(\frac{x+y}2, \partial C):\; y\in C,\; \|x-y\|\ge\epsilon
\bigr\}  \\
&\textstyle =\inf \bigl\{  
d(\frac{x+y}2, \partial C):\; y\in C,\; \|x-y\|=\epsilon
\bigr\} .
\end{aligned}
\end{equation}
The modulus $\delta_C(x,\cdot)$ is a natural generalization of the classical 
modulus of local uniform rotundity 
defined in the case when $C$ is the unit ball of $X$. For a general $C$, 
the equalities in \eqref{E:dc} are not obvious;
their proof, as well as some further properties of $\delta_C(x,\cdot)$, can be found in 
\cite{DEVE_loc_moduli} (see also \cite{DEVE_moduli} for some related results). We shall use the easy fact that 
one always has $\delta_C(x,\epsilon)\leq\epsilon/2$.

\begin{remark}\label{R:modulus}
\ 
\begin{itemize}
\item
It is clear that the function $\delta_C(x,\cdot)$ is non-decreasing
on $[0,2r)$, and $\delta_C(x,0)=0$. 
\item
Under the additional assumption that
$X$ is finite-dimensional, an easy compactness argument implies that
\\
{\em  if $C$ is rotund and $E\subset\partial C$ is a bounded set then $\inf_{x\in E}\delta_C(x,\epsilon)>0$
 for each $\epsilon\in(0,2r)$.}
\\
In particular:
\begin{enumerate}[(a)]
\item if $C$ is rotund then it is locally uniformly rotund (i.e., $\delta_C(x,\epsilon)>0$
for all $x\in\partial C$ and $\epsilon\in(0,2r)$);
\item if $C$ is rotund and bounded then it is uniformly rotund (i.e.,
$\inf_{x\in\partial C}\delta_C(x,\epsilon)>0$ for all $\epsilon\in(0,2r)$).
\end{enumerate}
\end{itemize}
\end{remark}

\begin{lemma}\label{L:primo}
Let $\mathrm{dim}\,X=d\ge2$, let $B\subsetneq C$ be two convex bodies in $X$, and let
$C$ be rotund. Then
$$
\mathrm{int}_C B\subset\mathrm{int}\, e(B).
$$
\end{lemma}

\begin{proof}
Consider an arbitrary $x\in\mathrm{int}_C B$. If $x\in\mathrm{int}\,C$ then clearly $x\in \mathrm{int}\, e(B)$
(since $x\in\mathrm{int}\,B$ and $B\subset e(B)$). 
So let $x\in\partial C$. Let $C$ contain a closed ball of radius $r>0$.
Then there exists $\epsilon\in(0,2r)$ such that $B(x,\epsilon)\cap C\subset B$,
and hence $d(x,\partial_C B)>\epsilon$. Notice that $\delta:=\delta_C(x,\epsilon)=\delta_B(x,\epsilon)>0$.

Now, for each $y\in\partial_C B$ and each $g\in\Sigma_1(y,B)$, we have $B(\frac{x+y}2,\delta)\subset B$,
and hence $g(\frac{x+y}2)+\delta\le\max g(B)=g(y)$. This easily implies that
$g(x)\le g(y)-2\delta$. Consequently, $B(x,2\delta)\subset e(B)$.
\end{proof}

\begin{proposition}\label{P:contenimento}
Let $\mathrm{dim}\,X=d\ge2$. Let $B_1\subsetneq B_2\subsetneq C$ be convex bodies in $X$,
and let the body $C$ be rotund, unbounded and without any asymptotic direction.
If $B_1 \subset \mathrm{int}_C B_2$ then 
$$e(B_1) \subset \mathrm{int}\,e(B_2).$$
\end{proposition}

\begin{proof}
Consider an arbitrary $x_0\in e(B_1)$. If $x_0\in B_1$ then $x_0\in \mathrm{int}_C B_2\subset \mathrm{int}\,e(B_2)$
(Lemma~\ref{L:primo}) and we are done. So let us assume that
$$
x_0\in e(B_1)\setminus B_1 \equiv e(B_1)\setminus C.
$$
By the Hahn-Banach theorem,
there exists $h\in S_{X^*}$ such that $m:=\inf h(C)>h(x_0)$.

By Lemma~\ref{L:Gamma}, the set $\Gamma_C(x_0)$ (defined therein) is bounded.
Thus there exists $\gamma>m$ such that $\Gamma_C(x_0)\subset [h\le \gamma]_C$.
In particular, $\Gamma_C(x_0)\cup\{x_0\}\subset [h\le\gamma]_X$.
By Lemma~\ref{L:bdd_sublevels}, $[h\le\gamma+1]_C$ is compact. 
Consider the halfspace $H:=[h\le \gamma+1]_X$.
Since $B_1\cap H$ is compact and disjoint from $\partial_C B_2$, we have
$d:=\mathrm{dist}(B_1\cap H, \partial_C B_2)>0$.

Fix an arbitrary $y\in \partial_C B_2$. Since $y\notin B_1$ and $x_0\in e(B_1)$, 
Lemma~\ref{L:5.3.bis} implies existence of
$y'\in B_1\cap (x_0,y)$. Since $C$ is rotund and $y,y'\in C$ are distinct, there exists
$y''\in(y',y)\cap\mathrm{int}\,C$. So on the half-line from $x_0$ through $y$, 
the points are put in this order:
\begin{equation}\label{E:punti}
x_0,\quad y',\quad y'', \quad y\,.
\end{equation}
Thus $(x_0,y'')\cap\partial C=\{z_1\}$ where $z_1\in(x_0,y']$. Since $x_0,y'\in e(B_1)$,
we have $z_1\in e(B_1)\cap\partial C\equiv B_1\cap\partial C$. 
It is clear from the ordering \eqref{E:punti} with $z_1\in(x_0,y'')$ that $x_0$
cannot belong to the cone $K(z_1,C)$.
By Lemma~\ref{L:good}, $z_1\in\mathrm{conv}\bigl[\Gamma_C(x_0)\cup\{x_0\}\bigr]$.
In particular, $h(z_1)\le\gamma$.

Fix some $0<\epsilon<\min\{d,1\}$. Since $\|y-z_1\|\ge d>\epsilon$,
there exists a unique $v\in(z_1,y)$ with $\|v-z_1\|=\epsilon$.
Notice that $h(v)\le h(z_1)+\epsilon<\gamma+1$. Since
$\delta_C(z_1,\epsilon)\ge\bar\delta:=\inf_{u\in \partial C\cap H} \delta_C(u,\epsilon)>0$, 
if we define $w=\frac{z_1+v}{2}$
 then we have $B(w,\bar\delta)\subset C$. Notice that
$$
d(w,\partial_C B_2)\ge d(z_1,\partial_C B_2)-\|w-z_1\| > \epsilon-(\epsilon/2)=\epsilon/2\ge\bar{\delta}.
$$
Consequently, $B(w,\bar\delta)\cap\partial_C B_2=\emptyset$. Since $w\in B_2$, we get
$B(w,\bar\delta)=B(w,\bar\delta)\cap C\subset B_2$. It follows that
the cone $K(y, B_2)$ contains the cone $c(y,B(w,\bar\delta))$. Since
$w\in(x_0,y)$, we obtain that $K(y,B_2)$ contains $B(x_0,\bar\delta)$.
Since $\bar\delta$ does not depend on $y$, we have
that $e(B_2)$ contains the
ball $B(x_0,\bar\delta)$. We are done.
\end{proof}

\begin{theorem}\label{T:ext_contQC}
Let $\mathrm{dim}\,X=d\ge2$. Let $C\subset X$ be an unbounded, rotund, convex 
body with no asymptotic directions.
Then each continuous QC function $f\colon C\to\R$ admits a continous QC
extension $F\colon X\to\R$.
\end{theorem}

\begin{proof}
For each $\alpha\in\R$, the set
$$
A_\alpha:=[f<\alpha]_C
$$
is convex and relatively open in $C$. Define $B_\alpha:=\overline{A}_\alpha$. Then
$$
B_\alpha \subset \mathrm{int}_C B_\beta \quad\text{whenever $\alpha<\beta$.}
$$
By Proposition~\ref{P:contenimento}, 
$$
e(B_\alpha)\subset\mathrm{int}\, e(B_\beta)\quad\text{whenever $\alpha<\beta$.}
$$
Since clearly $\bigcup_{\alpha\in\R}\mathrm{int}_C B_\alpha=C$, the set $B_\alpha$ 
has nonempty interior (even in $X$) for each sufficiently large
$\alpha$, and hence we can apply Proposition~\ref{P:riempimento} to conclude that
$$
X=\bigcup_{\alpha\in\R}e(B_\alpha)=\bigcup_{\alpha\in\R}\mathrm{int}[e(B_\alpha)].
$$
Moreover, since clearly $B_\alpha\subset [f\le\alpha]_C$, we have $\bigcap_{\alpha\in\R}B_\alpha=\emptyset$.
An easy application of Lemma~\ref{L:intersezione} gives $\bigcap_{\alpha\in\R}e(B_\alpha)=\emptyset$.
Therefore $\bigcap_{\alpha\in\R}\mathrm{int}_X [e(B_\alpha)]=\emptyset$.

Now, let us define $F\colon X\to\R$ by
$$
F(x)=\sup\bigl\{
\alpha\in\R: x\notin \mathrm{int}\,e(B_\alpha)
\bigr\}.
$$
By \cite[Proposition~2.5]{DEVEqc_UC}, $F$ is a continuous QC extension of $f|_{\mathrm{int}\,C}$ to the whole space.
By continuity, $F|_C=f$.
\end{proof}

\medskip

\section{Extension without continuity}

\begin{proposition}\label{P:solocontenimento}
	Let $\mathrm{dim}\,X=d\ge2$. Let $B_1\subset B_2\subset C$ be convex bodies in $X$,
	and let  $C$ contain no lines. Then 
	$e(B_1) \subset e(B_2)$.
\end{proposition}

\begin{proof}
	
	Suppose on the contrary that there exists $x\in e(B_1)\setminus e(B_2)$. Then:
	\begin{enumerate}
		\item $x\not\in C$ (by Remark~\ref{R:e});
		\item there exist $y\in\partial_C B_2$ and $f\in S_{X^*}$ such that 
		$$\theta:=\sup f(B_2)=f(y)<f(x).$$
	\end{enumerate}
We claim that $y\not\in B_1$. Indeed, if $y\in B_1$ then necessarily $y\in\partial_C B_1$ and $f\in K(y,B_1)$, and hence
$x\not\in e(B_1)$, a contradiction. Now, 
by Lemma~\ref{L:5.3.bis}  there exists $z\in (x,y)\cap B_1$.  In particular, we have 
	$$f(z)>\theta= \sup f(B_2)\geq\sup f(B_1)\geq f(z),$$
and this contradiction completes the proof.
\end{proof}

\begin{theorem}\label{T:ext_contQCsenza rotund}
	Let $\mathrm{dim}\,X=d\ge2$. Let $C\subset X$ be a convex 
	body with no asymptotic directions.
	Then each continuous QC function $f\colon C\to\R$ admits an upper semi-continuous QC
	extension $F\colon X\to\R$.
\end{theorem}

\begin{proof}
	For each $\alpha\in\R$, the set
	$$
	A_\alpha:=[f<\alpha]_C
	$$
	is convex and relatively open in $C$. Define $B_\alpha:=\overline{A}_\alpha$. Then
	$$
	B_\alpha \subset \mathrm{int}_C B_\beta \quad\text{whenever $\alpha<\beta$.}
	$$
	By Proposition~\ref{P:solocontenimento}, 
	$$
	e(B_\alpha)\subset e(B_\beta)\quad\text{whenever $\alpha<\beta$.}
	$$
	Since clearly $\bigcup_{\alpha\in\R}\mathrm{int}_C B_\alpha=C$, the set $B_\alpha$ 
	has nonempty interior (even in $X$) for each sufficiently large
	$\alpha$, and hence we can apply Proposition~\ref{P:riempimento} to conclude that
	$$
X=\bigcup_{\alpha\in\R}e(B_\alpha)=\bigcup_{\alpha\in\R}\mathrm{int}[e(B_\alpha)].
$$
The same argument as in the proof of Theorem~\ref{T:ext_contQC} gives 
$\bigcap_{\alpha\in\R}\mathrm{int}[e(B_\alpha)]=\emptyset$. 

	Now, let us define $F\colon X\to\R$ by
	$$
	F(x)=\sup\bigl\{
	\alpha\in\R: x\notin \mathrm{int}\,e(B_\alpha)
	\bigr\}.
	$$
	By the proof of \cite[Proposition~2.5]{DEVEqc_UC}, $F$ is a continuous upper semi-continuous QC extension of $f$ to the whole space.
\end{proof}

The following proposition shows that in the previous result we cannot relax the hypothesis about the continuity of $f$ 
by requiring just its upper semi-continuity.

\begin{proposition}\label{non_uscQC}
	Let $C=[0,1]\times[-1,1]\subset X:=\R^2$, then there exists a QC  upper semi-continuous function $f:C\to\R$ not admitting  QC  upper semi-continuous extension to the whole $X$.
\end{proposition}

\begin{proof}
	Consider the function $f$ on $C$ defined by
	$$f(x,y)=\begin{cases}
		0 & \text{if}\ (x,y)\in[0,1]\times[-1,0),\\
		y & \text{if}\  (x,y)\in(0,1)\times[0,1] ,\\
		1 &  \text{otherwise.}	
	\end{cases}$$
It is clear that $[f<\alpha]_C$ is empty for $\alpha\le0$, and coincides with $C$ for $\alpha>1$.
For $\alpha\in(0,1]$, we have $[f<\alpha]_C =\bigl([0,1]\times[-1,0)\bigr)\cup \bigl((0,1)\times[0,\alpha)\bigr)$.
Consequently, all strict
sublevel sets of $f$ are convex and relatively open in $C$, proving that $f$ is QC and  upper semi-continuous. 
Now, suppose on the contrary that there exists a QC upper semi-continuous function $F\colon X\to\R$ 
such that $F|_C=f$, and consider the open convex set $U=[F<\frac12]_X$. 
Since $F(0,-1)=f(0,-1)=0$, we have
$(0,-1)\in U$, and hence there exists $\epsilon>0$ such that $(0,-1)+\epsilon B_X\subset U$. Morover,
$U$ contains the relatively open segment $(0,1)\times\{\frac13\}$. So, by convexity of $U$, it contains also $(0,0)$.
But this is a contradiction since $F(0,0)=f(0,0)=1$.
\end{proof}

On the other hand, we have Theorem~\ref{T:ext_uscQC su aperto} below
about extension of QC  upper semi-continuous functions from {\em open} convex sets whose closure has no asymptotic directions.
For its proof we will need the following observation.

\begin{observation}\label{O:interni}
Let $\{D_n\}_n$ be a sequence of open convex sets in a normed space $X$. Suppose that
$X=\bigcup_n \overline{D}_n$, and $D_n\subset D_{n+1}$ for each $n$.
Then $X=\bigcup_n D_n$.
\end{observation}
\begin{proof}
The set $G:=\bigcup_n D_n$ is a dense, open convex set. By a well-known general fact,
$G=\mathrm{int}(\overline{G})=X$.
\end{proof}

\begin{theorem}\label{T:ext_uscQC su aperto}
	Let $\mathrm{dim}\,X=d\ge2$. Let $A\subsetneq X$ be a nonempty open convex set whose closure has no asymptotic directions.
	Then each upper semi-continuous QC function $f\colon A\to\R$ which is bounded on bounded sets admits an upper semi-continuous QC
	extension $F\colon X\to\R$ which is bounded on bounded sets.
\end{theorem}

\begin{proof} We proceed in a similar way as in the proof of Theorem~\ref{T:ext_contQCsenza rotund}. 
	For each $\alpha\in\R$ define
	$$
	A_\alpha:=[f<\alpha]_A, \quad B_\alpha:=\overline{A}_\alpha.
	$$
	Then: $A_\alpha\subset A_\beta$ whenever $\alpha<\beta$; $\bigcup_{\alpha\in\R}A_\alpha=A$; and
$A_\alpha\ne\emptyset$ for each sufficiently large $\alpha$. 
It is obvious that, for every fixed $\alpha$, either $B_\alpha=\emptyset$ or $B_\alpha$ is a convex body.
Consider the convex body $C:=\overline{A}$.

Let $r>0$ be such that $r B_X$ intersects $A$. Since $f$ is bounded above on bounded sets,
we must have $\mathrm{int}(r B_X\cap C)\subset A_\alpha$ for each sufficiently large $\alpha$.
For every such $\alpha$, $r B_X\cap C=\overline{\mathrm{int}(r B_X\cap C)}\subset\overline{A}_\alpha=B_\alpha$.
So we can apply Observation~\ref{O:riempimento} and Proposition~\ref{P:riempimento} to obtain that
$X=\bigcup_{\alpha\in\R}e(B_\alpha)$. 

By monotonicity of $\alpha\mapsto B_\alpha$ and by Proposition~\ref{P:solocontenimento},  
$e(B_\alpha)\subset e(B_\beta)$ whenever
$\alpha<\beta$. Hence $X=\bigcup_{\alpha\in\R}\mathrm{int}_X[e(B_\alpha)]$ by 
Observation~\ref{O:interni}.

We claim that $\bigcap_{\alpha\in\R}e(B_\alpha)=\emptyset$. To show this, consider an arbitrary $x\in C$
and an open ball $U$ centered at $x$. Since ${\overline\alpha}:=\inf f(U\cap A)\in\R$, we have $A_{\overline\alpha}\cap U=\emptyset$.
But then also $B_{\overline\alpha}\cap U=\emptyset$. Thus $\bigcap_{\alpha\in\R}B_\alpha=\emptyset$,
and our claim follows by Lemma~\ref{L:intersezione}. Obviously, $\bigcap_{\alpha\in\R}\mathrm{int}[e(B_\alpha)]=\emptyset$.

		Now, let us define $F\colon X\to\R$ by
	$$
	F(x)=\sup\bigl\{
	\alpha\in\R: x\notin \mathrm{int}\,e(B_\alpha)
	\bigr\}.
	$$
	By the proof of \cite[Proposition~2.5]{DEVEqc_UC}, $[F<\alpha]_X\subset\mathrm{int}\, e(B_\alpha)$
for each $\alpha\in\R$, and
$F$ is an upper semi-continuous QC extension of $f$ to the whole space.
Such a function is automatically bounded above on every compact set $K\subset X$.
Moreover, since $\bigcap_{\alpha\in\R}e(B_\alpha)=\emptyset$ and the sets $e(B_\alpha)$ are closed,
we must have $e(B_\beta)\cap K=\emptyset$ for some $\beta\in\R$.
Since by  \cite[Proposition~2.5]{DEVEqc_UC}, $[F<\beta]\subset e(B_\beta)$, it follows
that $\inf F(K)\ge\beta$. The proof is complete.
\end{proof}

%%%%%%%%%%%%%%%%%%%%%%%%%%%%%%%%%%%%%%%%%%%%%%%%
\medskip

\section{Characterizations of extendability}

In the present section, we state a few characterizations of extendability of QC functions
from a finite-dimensional closed convex set in an arbitrary normed space. 
They will follow easily
by combining together main results of the previous sections, our result from Theorem~\ref{T:old}(c), 
and the following two well-known facts: 
\begin{itemize}
\item  each finite-dimensional convex set $C$ has nonempty {\em relative interior}, 
that is, the interior of $C$ in its affine hull $\mathrm{aff}\,C$;
\item each finite-dimensional subspace $Y$ of a normed space is complemented, that is,
there exists a bounded linear projection onto $Y$.
\end{itemize}

\smallskip

Let us start with two trivial observations. 

\begin{observation}\label{O:1dim}
\begin{enumerate}[(a)]
\item Let $I\subsetneq\R$ be a (possibly degenerate) nonempty closed interval. 
Let $f\colon I\to\R$ be a  %\cd{qui ho aggiunto la richiesta di continuità su $f$} 
function, and let $F\colon\R\to\R$ be the extension of $f$ 
defined to be constant on the closure of each
of the (at most two) connected components of $\R\setminus I$.
\item Let $Y$ be a finite-dimensional subspace of a normed space $X$, and let
$P\colon X\to Y$ be a bounded linear projection onto $Y$. Let $f\colon Y\to\R$,
and consider its extension $F:=f\circ P$ to the whole $X$.
\end{enumerate}
\smallskip\noindent
In any of the above two cases,
if $f$ is QC and has one of the properties
``continuous'', ``uniformly continuous'', ``Lipschitz'', then $F$ is QC and has the same
property.
\end{observation}

Recall that the {\em dimension} of a convex set is defined as the dimension of its affine hull.

\begin{corollary}[Trivial cases]\label{C:trivial}
Let $C$ be a finite-dimensional
closed set in an arbitrary normed space $X$, and let $f\colon C\to\R$
be a QC function. If $C$ is either affine or at most one-dimensional, then
there exists a QC extension $F\colon X\to\R$ that shares any of the following
properties satistied by $f$: ``continuous'', ``uniformly continuous'', ``Lipschitz''.
\end{corollary}

\begin{proof}
If $C$ is affine, assume that $C$ is a subspace and apply Observation~\ref{O:1dim}(b).
If $C$ is at most one-dimensional, first use Observation~\ref{O:1dim}(a) to extend $f$ 
to a line (one-dimensional affine set) containing $C$, 
and then extend further to $X$ by using Observation~\ref{O:1dim}(b).
\end{proof}

\smallskip

For brevity of formulations we shall use the following terminology.

\begin{definition}
Let
$\mathcal{P}_1$ and $\mathcal{P}_2$ be two continuity-type properties of functions.
Let $C$ be a closed convex set in a normed space $X$.
Then the assertion
$$\text{\em ``$\mathcal{P}_1$ QC extends to $\mathcal{P}_2$ QC''}$$
means that each $\mathcal{P}_1$ QC function $f\colon C\to\R$ admits
a $\mathcal{P}_2$ QC extension $F\colon X\to\R$.
\end{definition}

\medskip

\begin{remark}[Lipschitz QC extensions]
Let $X$ be an arbitrary normed space and $C\subset X$ a (not necessarily finite-dimensional) 
closed convex set of dimension at least two.
\begin{enumerate}[(a)]
\item By our Theorem~\ref{no_lipQC}, {\em if $C$ is not affine and 
has nonempty relative interior
then the extendability property ``Lipschitz QC extends to Lipschitz QC'' fails for $C$.}
(Indeed, we can assume that $0\in C$; then $C$ is a convex body in the normed space
$Y=\mathrm{aff}\,C$.)
This means that a ``quasiconvex version'' of Theorem~\ref{T:McS} fails in a quite strong way.
\item On the other hand, {\em if $C$ is an affine set then the 
extendability property ``Lipschitz QC extends to Lipschitz QC'' holds
for $C$ even with preserving the Lipschitz constant.} (Indeed, we can suppose that
$C$ is a [not necessarily complemented] subspace, and then apply our previous result \cite[Theorem~3.4]{DEVEqc_UC}.)
\end{enumerate}
\end{remark}

\medskip

The following three theorems provide geometric characterizations
of finite-dimensional closed convex sets from which every Lipschitz QC function can be extended
to a QC function on the whole space with some weaker than Lipschitz
continuity property. Moreover, it turns out that
such extendability properties remain valid if we substitute 
``Lipschitz'' with ``uniformly continuous'' in the first one, and with
``continuous'' in the second and third one. Of course, we consider
only the cases which are not covered by Corollary~\ref{C:trivial}.

We shall say that a closed convex set $C$ is {\em relatively rotund} if it is rotund
in its affine hull (that is, if $x,y\in C$ are two distinct points then 
$\frac{x+y}2$ belongs to the relative interior of $C$).

\begin{theorem}[Uniformly continuous QC extensions]\label{T:char_uc}
Let $X$ be an arbitrary normed space, and
let $C\subset X$ be a finite-dimensional, closed convex, non-affine set of dimension at least two. 
For extendability from $C$ to $X$, the following assertions are equivalent.
\begin{enumerate}[(i)]
\item ``Uniformly continuous QC extends to uniformly continuous QC''.
\item ``Lipschitz QC extends to uniformly continuous QC''.
\item $C$ is bounded and relatively rotund.
\end{enumerate}
\end{theorem}

\begin{proof}
First, we can assume that $C$ contains the origin, and 
by Observation~\ref{O:1dim}(b), we can assume that $\mathrm{aff}\,C=X$
and hence $C$ is a convex body in $X$. Now, 
$(i)\Rightarrow(ii)$ is obvious. Let us prove the implication $(ii)\Rightarrow(iii)$.
Assume that $(iii)$ is false. If $C$ is not rotund, $(ii)$ is false by Theorem~\ref{T:nonR}.
If $C$ is rotund and unbounded, $(ii)$ is false by Theorems~\ref{noQC}~and~\ref{no_ucQC}.
Finally, if $(iii)$ holds then $C$ is uniformly rotund (Remark~\ref{R:modulus}) 
and hence $(i)$ holds by Theorem~\ref{T:old}(c).
\end{proof}

\smallskip

\begin{theorem}[Continuous QC extensions]\label{T:char_c}
Let $X$ be an arbitrary normed space, and
let $C\subset X$ be a finite-dimensional, closed convex, non-affine set of dimension at least two. 
For extendability from $C$ to $X$, the following assertions are equivalent.
\begin{enumerate}[(i)]
\item ``Continuous QC extends to continuous QC''.
\item ``Lipschitz QC extends to continuous QC''.
\item $C$ is rotund and has no asymptotic directions.
\end{enumerate}
\end{theorem}

\begin{proof}
As in the previous proof, we can assume that $C$ is a convex body in $X$.
Now, $(i)\Rightarrow(ii)$ is obvious. Let us prove the implication $(ii)\Rightarrow(iii)$.
Assume that $(iii)$ is false. Then $(ii)$ is false by theorems \ref{noQC}
and \ref{T:nonR}.
Finally, the implication $(iii)\Rightarrow(i)$ holds by Theorem~\ref{T:ext_contQC}.
\end{proof}

\smallskip

\begin{theorem}[Upper semicontinuous QC extensions]\label{T:char_qc}
Let $X$ be an arbitrary normed space, and
let $C\subset X$ be a finite-dimensional, closed convex, non-affine set of dimension at least two. 
For extendability from $C$ to $X$, the following assertions are equivalent.
\begin{enumerate}[(i)]
\item ``Continuous QC extends to (just) QC''.
\item ``Continuous QC extends to upper semicontinuous QC''.
\item ``Lipschitz QC extends to (just) QC''.
\item ``Lipschitz QC extends to upper semicontinuous QC''.
\item $C$ has no asymptotic directions.
\end{enumerate}
\end{theorem}

\begin{proof}
Again, we can assume that $C$ is a convex body in $X$.
Now, the implications $(ii)\Rightarrow(i)\Rightarrow(iii)$ and $(ii)\Rightarrow(iv)\Rightarrow(iii)$ are obvious. To show
$(iii)\Rightarrow(v)$, assume that $(v)$ is false and apply Theorem~\ref{noQC}.
Finally, the implication $(v)\Rightarrow(ii)$ holds by Theorem~\ref{T:ext_contQCsenza rotund}.
\end{proof}

\smallskip

Finally, let us recall that our Proposition~\ref{non_uscQC} shows
that the extension property ``upper semicontinuous QC extends to upper semicontinuous QC''
fails e.g.\ for a rectangle in $\R^2$.  

\bigskip

\paragraph*{\bf Acknowledgement}
The research of the first author was supported by the 
INdAM--GNAMPA, and by MICINN project PID2020-112491GB-I00 (Spain).
The research of the second author was supported by the INdAM--GNAMPA, and
by the University of Milan, Research Support Plan PSR 2024.

%%%%%%%%%%%%%%%%%%%%%%%%%%%%%%%%%%%%%%%%%%%%%%%%%%%%%%


\bigskip



\begin{thebibliography}{WW}


\bibitem{BFitzV} K.J.~Arrow and A.C.~Enthoven,
{\em Quasi-concave programming},
Econometrica \textbf{29} (1961), 779--800.

\bibitem{gencon} M.~Avriel, W.E.~Diewert, S.~Schaible, and I.~Zang,
{\em Generalized concavity},
Classics in Applied Mathematics, No. 63, reprint of the 1988 original, Society for Industrial and Applied Mathematics (SIAM), Philadelphia, 2010.

\bibitem{BalRep} M.V.~Balashov and D.~Repov\v{s}, 
{\em Uniform convexity and the splitting problem for selections}, 
J. Math. Anal. Appl. \textbf{360} (2009), 307--316.


\bibitem{DEVEqc_infdim}  C.A.~De Bernardi and L.~Vesel\'y,
{\em Extendability of continuous quasiconvex functions from subspaces},
J. Math. Anal. Appl.  \textbf{526}  (2023),   Paper No. 127277, 12 pp.

\bibitem{DEVEqc_examples} C.A.~De Bernardi and L.~Vesel\'y,
{\em Rotundity properties, and non-extendability of Lipschitz quasiconvex functions},
J. Convex Anal.  \textbf{30}  (2023),   329--342.

\bibitem{DEVEqc_UC} C.A.~De Bernardi and L.~Vesel\'y,
{\em On extension of uniformly continuous quasiconvex fucntions}, Proc. Amer. Math. Soc.  \textbf{151}  (2023),   1705--1716.



\bibitem{DEVE_moduli}  C.A.~De Bernardi and L.~Vesel\'y,
{\em Moduli of uniform convexity for convex sets.},  Arch. Math. \textbf{123} (2024),  413--422. 

\bibitem{DEVE_loc_moduli}  C.A.~De Bernardi and L.~Vesel\'y,
{\em Moduli of local uniform rotundity for convex bodies}, Can. Math. Bull. \textbf{68} (2025),
1251--1261.  




\bibitem{GaleKlee} D.~Gale and V.~Klee,
{\em Continuous convex sets},
Math. Scand. \textbf{7} (1959), 379--391.

\bibitem{Klee} V.~Klee,
{\em Asymptotes and projections of convex sets},
Math. Scand. \textbf{8} (1960), 356--362.

\bibitem{McShane} E.J.~McShane,
{\em Extension of range of functions},
Bull. Amer. Math. Soc. \textbf{40} (1934), 837--842.



\bibitem{pierskalla} H.~Greenberg and W.~Pierskalla,
{\em A Review of Quasi-Convex Functions},
Operations Research \textbf{19} (1971), 1553--1570.



\bibitem{LT-II} J.~Lindenstrauss and L.~Tzafriri,
{\em Classical Banach Spaces II. Function Spaces},
Ergebnisse der Mathematik und ihrer Grenzgebiete [Results in Mathematics and Related Areas], 97, Springer-Verlag, Berlin-New York, 1979.

\bibitem{Nordlander} G.~Nordlander,
{\em The modulus of convexity in normed linear spaces},
Ark. Mat. \textbf{4} (1960), 15--17.

\bibitem{penot} J.-P.~Penot,
{\em What is quasiconvex analysis?},
Optimization \textbf{47} (2000), 35--110.

\bibitem{Polyak} B.T.~Polyak,
{\em Existence theorems and convergence of minimizing sequences in extremum problems with restrictions},
Soviet Math. \textbf{7} (1966), 72--75.

\bibitem{Rock} T.R.~Rockafellar,
{\em Convex Analysis},
Princeton Mathematical Series,  Princeton University Press, Princeton, NJ,  1970.


\end{thebibliography}
\end{document}